%% file: 2d-fully_local.tex
\documentclass[reqno,final]{amsart}

\usepackage{caption}
\usepackage{subcaption}
\usepackage{tikz}
\usepackage{pgfplots}
\usepackage{verbatim}
\usepackage{amsfonts,latexsym,graphicx,amsmath,amssymb,bbm,enumitem,url}
\usepackage{algorithm}
\usepackage{algpseudocode}

\makeindex

\usepackage[color,notref,notcite]{showkeys}
\input{macros.tex}

\definecolor{labelkey}{rgb}{0.6,0,1}
\makeatletter
\def\BState{\State\hskip-\ALG@thistlm}
\makeatother

\begin{document}

	\title[]{A fully local hybridised second-order accurate scheme for advection-diffusion equations}
	
	\author{Hanz Martin Cheng}
\address{Department of Mathematics and Computer Science, Eindhoven University of Technology, P.O. Box 513, 5600 MB Eindhoven, The Netherlands.
	\texttt{h.m.cheng@tue.nl}}
	
	\date{\today}
	
	%
	%
	%
	\maketitle
		\begin{abstract}
			In this paper, we present a fully local  second-order upwind scheme, applicable on generic meshes. This is done by hybridisation, which is achieved by introducing unknowns on each edge of the mesh. By doing so, fluxes only depend on values associated to a single cell, and thus, this scheme can easily be applied even on cells near the boundary of the domain. Another advantage of hybridised schemes is that static condensation can be employed, leading to a very efficient implementation.  A convergence analysis, which also covers a flux-limited TVD variant of the scheme, is then presented. Numerical results are also given in order to compare this with a hybridised first-order upwind scheme and a classical cell-centered second-order upwind type scheme.
		\end{abstract}
	\section{Introduction}
	
		In this paper, we study a family of finite volume methods for stationary advection-diffusion equations. We start by presenting the choice of discretisation for the diffusive fluxes. This will be done via the hybrid mimetic mixed (HMM) method \cite{dro-10-uni}, which is equivalent to the SUSHI method \cite{EGH10-SUSHI}. For the advective fluxes, we propose a fully local second-order scheme. To motivate the problem, we start with a revision of the first-order upwind scheme, which is the easiest to implement.  One of the main disadvantages of the first-order upwind scheme is that it may easily introduce too much numerical diffusion into the solution of the system \cite{H06-bookIntroNum}. High-order schemes have been proposed in order to help mitigate the introduction of too much numerical diffusion. However, in order to achieve a high-order discretisation, more degrees of freedom (DOFs) are needed, which translates into a higher computational cost. In this work, we focus on second-order schemes. Classical second-order upwind schemes on Cartesian meshes \cite{PVW65-orig-upw} involve a 9-point stencil; hence, some  interpolation techniques or introduction of ghost cells are needed in order to apply them on cells near the boundary of the domain. 
		
		The novelty of this work is the introduction of a hybridised fully local second-order scheme, inspired by the ideas in \cite{VDM11-adv-diff}. To do so, in addition to unknowns at the cell center, we introduce one additional unknown to each cell face. This results to more unknowns in the system: now having $\sharp$ of cells $+$ $\sharp$ of faces (edges) unknowns, as compared to $\sharp$ of cells for cell-centered schemes. However, this allows us to reduce the dependence of the fluxes on values from neighboring cells. In the case of square cells (see Figure \ref{fig.CS}), we only need one cell value, and four interface values. Since we remove the direct dependence of the fluxes on the values from neighboring cells, this allows us to directly apply this hybridised second-order scheme even near the boundaries of the domain. Moreover, static condensation can be employed to make the implementation more efficient. Following this, we then present some convergence results, which cover, for the advective component, the classical second-order upwind type schemes, and also include some nonlinear total variation diminishing (TVD) methods \cite{BM-2dFV,L10-monotoneFV}. 
	\begin{figure}[h!]
	\includegraphics[width=0.15\linewidth]{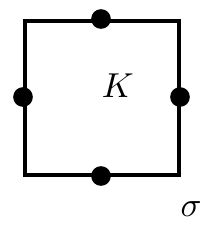} 
	\caption{Values needed for computing a hybridised second-order flux for edge $\sigma$.}
	\label{fig.CS}
\end{figure}

 The paper is organised into the following sections: We start by considering a stationary advection-diffusion equation, and present it in its finite volume form. Afterwards, we provide a discretisation of the diffusive fluxes, which will be done via the HMM method. We then give a short review of the cell-centered first and second-order upwind schemes on Cartesian meshes and give their extension onto generic polygonal meshes. Following this, we propose a hybridised fully local second-order scheme, applicable on generic meshes. Here, we show that by introducing additional unknowns along the faces (edges) of each cell, we can create a stencil that does not depend on values from neighboring cells. Convergence results will then be presented. Numerical tests will then be performed to illustrate the accuracy of the scheme, and also to compare this with a hybridised first-order upwind scheme and a classical cell-centered second-order upwind type scheme.

\section{The model problem}

	Consider the following stationary, linear scalar advection-diffusion equation with Dirichlet boundary conditions on a polytopal domain $\O$: Find $c\in H^1(\O)$ such that 
\begin{equation}\label{eq:model}
\begin{aligned}
\nabla\cdot(-\Lambda \nabla c + c\V) &= f \qquad \mathrm{ in} \quad \O, \\
c &= g \qquad \mathrm{ in } \quad \partial\Omega.
\end{aligned}
\end{equation}
Here, $\Lambda$ is a diffusion tensor, $\V$ is a velocity field, and $f$ is a source term. We begin by stating the assumptions on the data:
\begin{enumerate}[label=\combineln{A}{\arabic*}]
	\item \label{assum.Diff} $\Lambda$ is a measurable function from $\O$ to the set of $d \times d$ symmetric positive definite matrices, and there exists $\underline{\lambda}, \overline{\lambda} >0 $ such that, for a.e. $x\in \O$, the eigenvalues of $\Lambda(x)$ are in $[\underline{\lambda}, \overline{\lambda}]$;
	\item \label{assum.Vel} $\V \in C^1(\overline{\O})^d$ with $\divg (\V) \geq 0$;
	
	\item \label{assum.Src} $f\in L^2(\O)$.
\end{enumerate}

\subsection{Finite volume discretisation of the advection-diffusion equation} We now write a finite volume discretisation of the model \eqref{eq:model}. First, we define a mesh in the simplest intuitive way: a partition of $\O$ into polygonal (in 2D) or polyhedral (in 3D) sets. Following the notations in \cite[Definition 7.2]{GDMBook16}, we denote $\mathcal{T} = (\mesh; \edges)$ to be
the set of cells $K$ and faces (edges in 2D) $\sigma$ of our mesh, respectively. For each cell $K\in \mesh$, we denote by $|K|$ its $d$ dimensional measure, $\mathrm{diam}(K)$ its diameter, $\edges_K \subset \edges$ the set of faces (edges) of cell $K$, and $\x_K$ the cell center of gravity. The collection of faces is a disjoint union of two sets, $\edges = \edges_{\mathrm{int}} \cup \edges_{\mathrm{ext}}$, where $\edges_{\mathrm{int}}$ and $\edges_{\mathrm{ext}}$ denote the set of interior and exterior faces, respectively. For each interior face $\sigma\in\edges_K$, we denote by $K'$ the cell that shares the face $\sigma$ with $K$. Also, for $\sigma\in\edges$, we denote by $|\sigma|$ its $d-1$ dimensional measure, $\x_\sigma$ its center, and $\mathbf{n}_{K,\sigma}$ its normal direction pointing out of $K$. We also denote by $d_{K,\sigma}$ the orthogonal distance from $\x_K$ to $\sigma$ (see Figure \ref{fig.notations_2D_mesh}).

Denoting by $\mesh_h$ the mesh such that, $\max_{K\in\mesh_h}(\mathrm{diam}(K))=h$, the analysis performed in Section \ref{sec:Convergence} will require the following assumptions on the mesh.
\begin{enumerate}[label=\combineln{MR}{\arabic*}]
	\item \label{assum. mesh_star} Every cell $K\in\mesh_h$ is star-shaped with respect to $\x_K$.
	\item \label{assum. mesh_reg} The mesh regularity parameter, which is defined as 
	\begin{equation}\nonumber
	\mathrm{regul}(\mesh_h) := \max\bigg( \max_{\sigma \in \edges_{h,\mathrm{int}}} \frac{d_{K,\sigma}}{d_{K',\sigma}},  \max_{\sigma \in \edges_{h,\mathrm{ext}}} \frac{\mathrm{diam}(K)}{d_{K,\sigma}}, \max_{K\in\mesh_h} \mathrm{card}(\edges_K)\bigg),
	\end{equation}
	where $\mathrm{card}(\edges_K)$ denotes the number of edges of a cell $K$, is uniformly bounded as $h\rightarrow 0$.
\end{enumerate} 

\begin{figure}[h]
	\centering
	\includegraphics[width=0.3\linewidth]{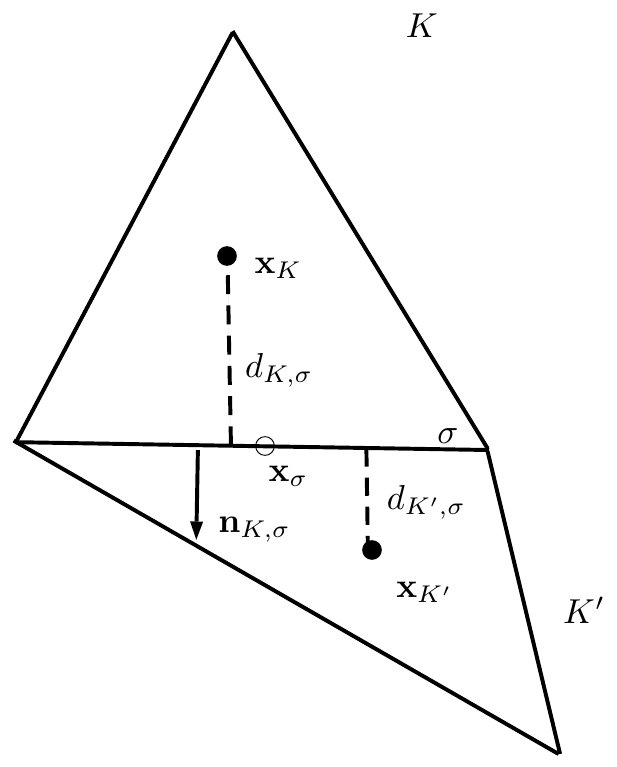} 
	\caption{Notations in a mesh cell.}\label{fig.notations_2D_mesh}
\end{figure}

We now denote the unknowns of the system, one on each cell and one on each edge by
\[
X_\disc:= \{\big((w_K)_{K\in\mesh},(w_\sigma)_{\sigma\in\edges}, w_K,w_\sigma \in \R \big)\}.
\]
Now, we take the integral of \eqref{eq:model} over a cell $K$, and use Gauss' Theorem to obtain
\[ \sum_{\edge\in\edgescv} \int_\sigma(-\Lambda \nabla c + c\V) \cdot \mathbf{n}_{K,\sigma}\,\,ds = \int_K f\,\,d\x.\]
 Following the ideas in \cite{VDM11-adv-diff}, we write $F_{K,\sigma}^D \approx \frac{1}{|\sigma|}\int_\sigma \Lambda \nabla c \cdot \mathbf{n}_{K,\sigma}\,\,ds$ and $F_{K,\sigma}^A \approx \frac{1}{|\sigma|}\int_\sigma c\V \cdot \mathbf{n}_{K,\sigma}\,\,ds$ approximations of the average diffusive and advective fluxes along $\sigma$, respectively. The discrete equation for flux balance in each cell $K$ reads
\begin{equation}\nonumber
\begin{aligned}
\sum_{\edge \in \edgescv}|\sigma|(F_{K,\sigma}^D+F_{K,\sigma}^A)= |K| f_K,
\end{aligned}
\end{equation}
where $f_K$ is the average value of the source term at cell $K$.
Here, we see that the main ingredients for obtaining an approximation for the solution $c$ is the definition of the diffusive and advective fluxes $F_{K,\sigma}^D$ and $F_{K,\sigma}^A$. 
\section{Diffusive fluxes}
In this section, we discuss the discretisation of the diffusive fluxes. This will be done via the hybrid-mimetic-mixed (HMM) method \cite{dro-10-uni}, for which the diffusive fluxes are based on the bilinear form $a(u,w):=\int_\O \Lambda \nabla u \cdot \nabla w \,\,d\x$, which stems from the weak formulation of the advection-diffusion equation \eqref{eq:model}. In particular, we define the diffusive fluxes such that for $c\in X_\disc$, we have for all $K\in\mesh, v\in X_\disc$, 
\begin{equation}\label{eq:diff_fluxes}
\sum_{\sigma\in\edges_K} |\sigma|F_{K,\sigma}^D (v_K-v_\sigma) = \int_K \Lambda_K \nabla_{\disc} c \cdot \nabla_{\disc} v \,\,d\x, 
\end{equation}
where $\Lambda_K$ is an approximation of $\Lambda$ at cell $K$, and $\nabla_{\disc}$ is a \emph{stabilised discrete gradient}. The discrete gradient is stabilised in the sense that it consists of a consistent term, added to a stabilisation term. The consistent term is linearly exact, and is defined via a cellwise discrete gradient $\ograd_\disc$, such that for all $q\in X_\disc$ and for all $K\in\mesh$, $(\ograd_\disc q)_{|K} = \ograd_{K} q$, where 
\begin{equation}\label{eq:discGrad_expression}
\ograd_{K} q := \frac{1}{|K|}\sum_{\sigma\in\edges_K} |\sigma| (q_\sigma-q_K) \mathbf{n}_{K,\sigma}, \
\end{equation}
and $\mathbf{n}_{K,\sigma}$ is the unit outward normal vector of $\sigma$. In order to ensure the coercivity of the diffusive flux, a stabilisation term needs to be added to the discrete gradient \eqref{eq:discGrad_expression}.  For the HMM method, the stabilisation term is defined such that for all $q\in X_\disc$ and for all $K\in\mesh$, $S_K: X_\disc \rightarrow L^2(K)$ is given by
\begin{equation} \label{eq:stab_K}
S_K (q) := \sum_{\sigma\in\edges_K} S_{K,\sigma} \mathbbm{1}_{D_{K,\sigma}}, 
\end{equation}
where $(D_{K,\sigma})_{\sigma\in\edges_K}$ are convex hulls of $\sigma$ and $\x_K$ (see Figure \ref{fig.notations_2D}), and 
\begin{equation}\label{eq:stab_Ksig}
S_{K,\sigma} := \frac{\sqrt d}{d_{K,\sigma}}[q_{\edge}-q_{\cv}-\ograd_{\cv}q \cdot(\x_\sigma-\centercv)] \mathbf{n}_{K,\sigma}.
\end{equation}
Here, $d_{K,\sigma}$ is the orthogonal distance from $\x_K$ to $\sigma$, and $\x_\sigma$ is the center of the edge $\sigma$.
\begin{figure}[h]
	\centering
	\includegraphics[width=0.3\linewidth]{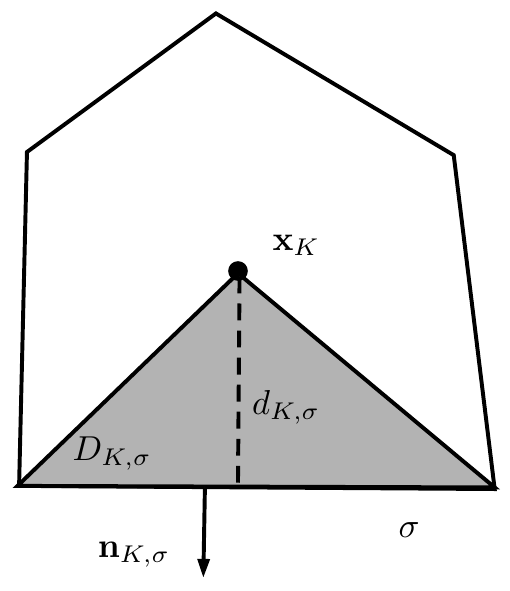} 
	\caption{Notations on cell $K$.}
	\label{fig.notations_2D}
\end{figure}

The stabilised discrete gradient is then defined such that for all $q\in X_\disc$ and for all $K\in\mesh$
\begin{equation}\label{eq:stabGrad}
(\nabla_{\disc} q)_{|K} := \ograd_{K} q + S_K (q).
\end{equation}

One important property of the stabilisation term, which will be needed to establish the coercivity of the scheme, is the following orthogonality condition.
\begin{lemma}[Orthogonality of the stabilisation term.] \label{lem:stab_orth} Let $q\in X_\disc$ and let $K\in\mesh$. Then, the stabilisation $S_K: X_\disc \rightarrow L^2(K)$ defined as in \eqref{eq:stab_K} satisfies the following orthogonality condition.
\begin{equation}\label{eq:stab_ortho}
\int_K S_K q \cdot \phi\,\,d\x = 0 \quad \forall \phi \in \R^d.
\end{equation} 
\end{lemma}

A proof of Lemma \ref{lem:stab_orth} can be found in \cite[Chapter 13]{GDMBook16}, but for completeness, we present an alternative proof below.

\begin{proof}
	Since $S_K q$ is piecewise constant with value $S_{K,\sigma}$ on each convex hull $D_{K,\sigma}$, we can write 
	\begin{equation}\nonumber
	\begin{aligned}
	\int_K S_{K} q \cdot \phi \,\,d\x &= \sum_{\sigma\in\edges_K} \int_{D_{K,\sigma}} S_{K,\sigma} \cdot \phi \,\,d\x\\
	&= \sum_{\sigma\in\edges_K} |D_{K,\sigma}| S_{K,\sigma} \cdot \phi.
	\end{aligned}
	\end{equation}
	Now, using the definition of $S_{K,\sigma}$ in \eqref{eq:stab_Ksig} and the geometric relation 
	\begin{equation} \label{eq:area_DKsig} 
	|D_{K,\sigma}| = \frac{1}{d} |\sigma| d_{K,\sigma},
	\end{equation}
	we get 
	\begin{equation}\nonumber
	\begin{aligned}
	\sum_{\sigma\in\edges_K} |D_{K,\sigma}| S_{K,\sigma} \cdot \phi &= \frac{\sqrt{d}}{2} \sum_{\sigma\in\edges_K} |\sigma|[q_{\edge}-q_{\cv}-\ograd_{\cv}q \cdot(\x_\sigma-\centercv)] \mathbf{n}_{K,\sigma} \cdot \phi\\
	&= \frac{\sqrt{d}}{2} \sum_{\sigma\in\edges_K} |\sigma|(q_{\edge}-q_{\cv})\mathbf{n}_{K,\sigma} \cdot \phi\\
	&\,\,-\frac{\sqrt{d}}{2}\sum_{\sigma\in\edges_K}|\sigma|[\ograd_{\cv}q \cdot(\x_\sigma-\centercv)] \mathbf{n}_{K,\sigma} \cdot \phi.
	\end{aligned}
	\end{equation}
    Using \eqref{eq:discGrad_expression}, we obtain
	\begin{equation}\label{eq:stab_orthoT1}
 \sum_{\sigma\in\edges_K} |\sigma|(q_{\edge}-q_{\cv})\mathbf{n}_{K,\sigma} \cdot \phi =  |K| \ograd_{K} q \cdot \phi.
	\end{equation}
	We then note that $\ograd_K q$ is a constant vector, and hence, $\ograd_{K} q \cdot (\x-\x_K)$ is a polynomial of degree one. This means that if $\x_\sigma$ is the edge midpoint of $\sigma$, then 
	\[
	|\sigma|[\ograd_{\cv}q \cdot(\x_\sigma-\centercv)] \mathbf{n}_{K,\sigma} \cdot \phi = \int_\sigma \ograd_{K} q \cdot (\x-\x_K) \mathbf{n}_{K,\sigma} \cdot \phi \,\,ds.
	\]
	Taking the sum over $\sigma\in\edges_K$, using Green's Theorem and the fact that $\phi$ is constant, we then get 
	\begin{equation}\nonumber
	\begin{aligned}
	\sum_{\sigma\in\edges_K} |\sigma|[\ograd_{\cv}q \cdot(\x_\sigma-\centercv)] \mathbf{n}_{K,\sigma} \cdot \phi &= \sum_{\sigma\in\edges_K} \int_\sigma \ograd_{K} q \cdot (\x-\x_K) \mathbf{n}_{K,\sigma} \cdot \phi \,\,ds
	\\ &= \int_K \divg(\phi\ograd_{K} q \cdot (\x-\x_K)) \,\,d\x\\
	&=\int_K \ograd_{K} q \cdot \phi \,\,d\x,
	\end{aligned}
	\end{equation}
 or equivalently,
 \begin{equation}\label{eq:stab_orthoT2}
 \sum_{\sigma\in\edges_K} |\sigma|[\ograd_{\cv}q \cdot(\x_\sigma-\centercv)] \mathbf{n}_{K,\sigma} \cdot \phi	= |K| \ograd_{K} q \cdot \phi.
 \end{equation}
The proof is then concluded by combining the expressions \eqref{eq:stab_orthoT1} and \eqref{eq:stab_orthoT2}.
\end{proof}


\section{Advective fluxes}
\subsection{Cell-centered first and second-order upwind fluxes}
In this section, we discuss the advective fluxes, starting with the standard first order upwind scheme for Cartesian meshes. For each control volume $K$, we assign one discrete unknown $c_K$, which approximates the average value of $c$ at the center $\x_K$ of $K$. Consider now a cell $K$ with eastern edge $\sigma$. Adapting the compass notation, the cells to the west and east of $K$ are denoted by $W, E$ respectively. The cell to the east of $E$ is then denoted by $EE$ (see Figure \ref{fig.SI_Cart}). 

	\begin{figure}[h!]
	\includegraphics[width=0.45\linewidth]{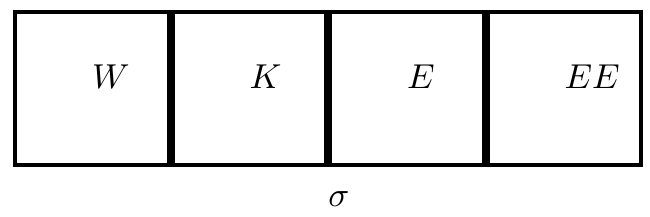} 
	\caption{Cells involved in computing a second-order upwind flux for edge $\sigma$.}
	\label{fig.SI_Cart}
\end{figure}

Denoting by $\x_\sigma$ the midpoint of the edge $\sigma$ and using the midpoint rule for computing the integral $\int_\sigma c\V \cdot \mathbf{n}_{K,\sigma} \,\,ds$, the upwind fluxes are then given by
\begin{equation}\label{eq:1up}
F_{K,\sigma}^A= c_K\V_{K,\sigma}^+ - c_E \V_{K,\sigma}^-,
\end{equation} where $\V_{K,\sigma} = \frac{1}{|\sigma|}\int_\sigma \V \cdot \mathbf{n}_{K,\sigma} \,\,ds$,$\V_{K,\sigma}^+ = \max(\V_{K,\sigma},0)$,$\V_{K,\sigma}^- = \max(-\V_{K,\sigma},0)$, and the value $c_\sigma=c(\x_\sigma)$ is approximated from the upwind direction.

 Essentially, this tells us that if material is flowing out (into) cell $K$ through the edge $\sigma$, corresponding to $\V \cdot \mathbf{n}_{K,\sigma}$ being positive (negative), then the value of $c_\sigma$ is approximated by $c_K$ ($c_E$), which makes sense since the material traveling through $\sigma$ comes from cell $K$ ($E$). From this, we see that the first-order upwind scheme has a very natural physical interpretation, and is fairly simple to implement. Moreover, the formulation \eqref{eq:1up} can straightforwardly be extended onto generic meshes, by treating $E$ as a generic neighboring cell that shares $\sigma$ with $K$. However, the upwind scheme may introduce too much numerical diffusion. Several methods can be employed in order to mitigate this, but for this work, we focus on second-order methods.

	We now discuss the second-order upwind scheme. Again, we look at the eastern edge $\sigma$ of cell $K$. The second-order upwind scheme approximates the advective flux $F_{K,\sigma}$ in the following manner:
\begin{equation}\label{eq:2up}
F_{K,\sigma}^A = 
\bigg(c_K+\frac{1}{2}(c_K-c_W)\bigg)\V_{K,\sigma}^+  -
\bigg(c_E+\frac{1}{2}(c_{E}-c_{EE})\bigg) \V_{K,\sigma}^-.
\end{equation}
Here, we see that the main difference between the second-order upwind flux \eqref{eq:2up} and the first-order upwind flux is that \eqref{eq:1up} approximates the value of $c_\sigma$ by a first-order Taylor expansion centered at $\x_K$ ($\x_E$) when $\V \cdot \mathbf{n}_{K,\sigma}$ is positive (negative), whereas \eqref{eq:2up} gives a second-order Taylor expansion in the same direction. Another thing we notice is that due to the presence of $-c_W$ and $-c_E$ in the second-order upwind flux \eqref{eq:2up}, it is no longer guaranteed that the discrete maximum principle is satisfied. This is typically the case for linear schemes which are second-order or higher. We now note that
\[\frac{1}{2}(c_K-c_W) \approx \frac{\partial c}{\partial x} (\x_K) |\x_\sigma-\x_K| = \nabla c \cdot (\x_\sigma-\x_K),
\]
and similarly
\[\frac{1}{2}(c_{EE}-c_E) \approx \frac{\partial c}{\partial x} (\x_E)|\x_\sigma - \x_E| = \nabla c \cdot (\x_\sigma - \x_E).
\]
Here, the gradient of $c$ is approximated from the upwind direction. To be specific, we see that if $\V_{K,\sigma}>0$, then $\frac{\partial c}{\partial x} (\x_K)$ is approximated by values from cells $K$ and $W$, which are located along the upwind direction. 

To generalise, second-order fluxes are obtained from first-order upwind fluxes by adding a correction term, where the correction term is related to a discrete gradient. In particular, the fluxes $F_{K,\sigma}^A$ for second-order schemes can be written in the following manner:


\begin{equation}\label{eq:2up_gen}
F_{K,\sigma}^A = 
\bigg(c_K+\widetilde{\nabla}_{\disc} c_K \cdot  (\x_\sigma-\x_K) \bigg)\V_{K,\sigma}^+ -
\bigg(c_E+\widetilde{\nabla}_{\disc} c_E \cdot (\x_\sigma-\x_E)\bigg)\V_{K,\sigma}^-,
\end{equation}
where $\widetilde{\nabla}_{\disc} c_K$ is a discrete gradient which approximates the value of $\nabla c$ at cell $K$.
Different choices on how to reconstruct the discrete gradient $\widetilde{\nabla}_{\disc}c$ then leads to  different schemes. For example, we can take $\widetilde{\nabla}_{\disc} c_K$ to be $\ograd_K c$ defined as in \eqref{eq:discGrad_expression}. For schemes with cell-centered unknowns, the value of $c_\sigma$ in \eqref{eq:discGrad_expression} needs to be chosen. Taking the value of $c_\sigma$ from the upwind direction then leads to the second-order upwind scheme. On the other hand, taking the value of $c_\sigma$ from the downwind direction leads to the second-order centered scheme. To extend the description of the fluxes \eqref{eq:2up_gen} onto non-Cartesian meshes,  we treat $E$ as a generic neighboring cell that shares the edge $\sigma$ with $K$. As with Cartesian meshes, using a second-order cell-centered scheme leads to a much wider stencil. The main point to consider is how the term $\widetilde{\nabla}_{\disc} c_K$ in \eqref{eq:2up_gen} is computed on generic meshes. Here, we use  $\ograd_K c$ defined as in \eqref{eq:discGrad_expression} as our discrete gradient. This gives us 
\begin{equation}\label{eq:discGrad_upwind}
\widetilde{\nabla}_{\disc} c_K = \frac{1}{|K|} \sum_{\edge \in \edgescv} |\sigma|(c_\sigma^{\mathrm{ up }}-c_K) \mathbf{n}_{K,\sigma},
\end{equation}
where the value of $c_\sigma^{\mathrm{ up }}$ is then taken from the upwind direction. That is, we take $c_\sigma^{\mathrm{ up }} = c_K$ if $\V_{K,\sigma}>0$ and $c_\sigma^{\mathrm{ up }}$ from the neighboring cell otherwise. 

As an example, consider a cell $K$ of a triangular mesh with edges $\sigma_1,\sigma_2,\sigma_3$ being shared with cells $E,W,S$ respectively. Taking a velocity field $\V$ such that $\V_{K,\sigma_1}>0$ and $\V_{K,\sigma_i}<0$ for the other edges $\sigma_i, i = 2,3$ of $K$ (see Figure \ref{fig.2dup_tri}), the discrete derivative is then computed to be 
\[
\widetilde{\nabla}_{\disc} c_K = \frac{1}{|K|} \big( |\sigma_2|(c_W-c_K) \mathbf{n}_{K,\sigma_2} + |\sigma_3|(c_S-c_K) \mathbf{n}_{K,\sigma_3}\big).
\]

\begin{figure}[h]
	\centering
	\includegraphics[width=0.5\linewidth]{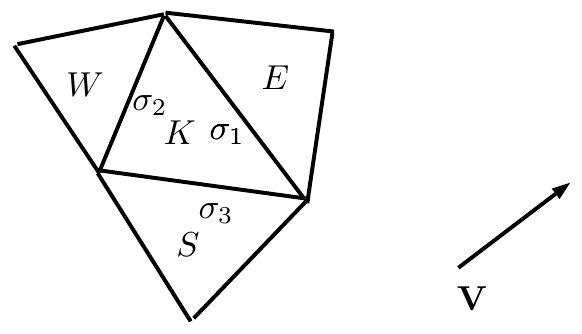} 
		\caption{Cell values needed for constructing $\ograd_K c$ for cell-centered schemes.}
	\label{fig.2dup_tri}
\end{figure}

Since the construction of the advective fluxes $F_{K,\sigma}^A$ on an edge $\sigma$ of cell $K$  in \eqref{eq:2up_gen} requires a linear reconstruction both in cell $K$ and from its neighboring cell, this requires information not only from the neighboring cells of $K$, but also from the neighbors of its neighbors. Moreover, this generalisation of the second-order upwind fluxes onto arbitrary polygonal meshes comes with an additional computational cost/storage. In particular, the sign of $\V \cdot \mathbf{n}_{K,\sigma}$ needs to be computed/stored for each of the edges of cell $K$ and its neighbors in order to determine the value of $c_\sigma^{\mathrm{ up }}$ in computing the discrete gradient \eqref{eq:discGrad_upwind}. 


\subsection{Fully local second-order fluxes}
In this section, we introduce the idea of having a fully local second-order scheme. In order to do so, we employ the concept of hybridisation, as in \cite{AB85-MFEM}. This was done for first order upwind schemes in \cite{VDM11-adv-diff} by the introduction of additional unknowns, one on each face, leading to the hybridised first-order upwind fluxes

\begin{equation}\label{eq:hybrid_firstOrder}
F_{K,\sigma}^A = c_K\V_{K,\sigma}^+ - c_\sigma \V_{K,\sigma}^-.
\end{equation}
 By doing so, the direct dependence of the fluxes on values from neighboring cells is eliminated. The main idea for this section is to extend the hybridised first-order upwind fluxes into second-order fluxes by writing 

\begin{equation} \label{eq:hybrid_flux} 
F_{K,\sigma}^A = \bigg(c_K + \widetilde{\nabla}_{\disc} c_K\cdot (\x-\x_K)\bigg)\V_{K,\sigma}^+ - c_\sigma \V_{K,\sigma}^-.
\end{equation}

One of the main advantages of using a hybridised scheme is, if information is expected to have arrived from another cell, we keep the value of $c$ at $\sigma$ to be implicit. Moreover, for second-order schemes, if the discrete gradient $\widetilde{\nabla}_\disc c$ is taken such that $(\widetilde{\nabla}_\disc c)_{|K}=\ograd_{K} c$ as defined in \eqref{eq:discGrad_expression}, then the values of $c_\sigma$ in the discrete gradient are determined naturally, without having to choose between the values in the upwind and downwind direction. This also allows us to save in terms of the computational cost associated with computing $(\widetilde{\nabla}_\disc c)$.
 
On the other hand, since we introduced additional unknowns along the edges, we also need to introduce additional equations (one corresponding to each edge). Denoting by $N_K$ and $N_e$ denote the number of cells and edges in the mesh, respectively, we then have a scheme with $N_K + N_e$ equations and $N_K + N_e$ unknowns. Using definitions \eqref{eq:diff_fluxes} and \eqref{eq:hybrid_flux} for the diffusive and advective fluxes, and denoting 
\[
F_{K,\sigma} = F_{K,\sigma}^D + F_{K,\sigma}^A,
\] the first $N_k$ equations are given by the balance of fluxes:
\begin{subequations}
	\begin{equation}\label{eq:scheme_fluxbal}
	\sum_{\edge \in \edgescv}|\sigma|F_{K,\sigma}= |K| f_K.
	\end{equation}
	Following this,  we  impose the conservation of fluxes for interior edges. That is, for cells $K$ and $L$ that share a common edge $\sigma$,
	\begin{equation}\label{eq:fluxcons}
	F_{K,\sigma}+F_{L,\sigma} = 0.
	\end{equation}
	Finally, on each boundary edge, an equation is needed for imposing the Dirichlet boundary conditions. In particular, for each $\sigma\in\edges_{\mathrm{ext}}$, we take $c_\sigma$ to be the average value of $g$ over $\sigma$, given by
	\begin{equation}\label{eq:bdcond}
	c_\sigma = \frac{1}{|\sigma|}\int_\sigma g \,\,ds. 
	\end{equation}
\end{subequations}

Note here that the scheme \eqref{eq:scheme_fluxbal}-\eqref{eq:bdcond}  consists of $N_e$ more equations than cell-centered schemes. The main advantage, however, of the hybridised formulation is that the fluxes no longer depend directly on values from neighboring cells, which enables us to use these second-order fluxes even on cells at the boundary of the domain. Moreover, we can employ static condensation for solving the system of equations, which essentially leads to solving only a system of $N_e$ equations in $N_e$ unknowns. 
\section{Convergence of the scheme} \label{sec:Convergence}
In this section, we study the convergence of a family of hybridised second-order finite volume schemes, with diffusive and advective fluxes as in \eqref{eq:diff_fluxes} and \eqref{eq:hybrid_flux}, respectively. Since the diffusive fluxes used here come from the HMM method, the convergence analysis for these terms are more or less the same as those found in \cite{VDM11-adv-diff,dro-10-uni,EGH10-SUSHI}. For our proofs, we focus on the advective fluxes, and how to deal with the correction term. For simplicity of exposition, we assume that the discrete gradient for the correction term in \eqref{eq:hybrid_flux} is taken from \eqref{eq:stabGrad}. A typical norm used for measuring errors is a discrete $L^2$ norm. For this, we start by defining a function reconstruction $\Pi_{\disc_h}: X_{\disc_h} \rightarrow L^2(\O)$, such that for all $q_h\in X_{\disc_h}$
\begin{equation}\label{eq:Pidc}
(\Pi_{\disc_h} q_h)_{|K} = q_K, \quad \forall K\in\mesh_h.
\end{equation} The discrete $L^2$ norm is then defined as
\begin{equation}\label{eq:discL2}
\norm{q_h}{L^2(\O)} := \norm{\Pi_{\disc_h} q_h}{L^2(\O)}=\bigg(\sum_{K\in\mesh_h} |K| q_K^2\bigg)^{1/2}, \quad \forall q_h \in X_{\disc_h}.
\end{equation}  Following the ideas in \cite{VDM11-adv-diff,GDMBook16}, it is also useful to perform analysis for hybridised schemes by using a 
discrete $H^1$-like norm on $X_{\disc_h}$:
\begin{equation}\label{eq:discH1}
\norm{q_h}{1,\disc_h} = \bigg(\sum_{K\in\mesh_h}\sum_{\sigma\in\edges_K}\frac{|\sigma|}{d_{K,\sigma}}|q_K-q_\sigma|^2\bigg)^{1/2}, \quad \forall q_h \in X_{\disc_h}.
\end{equation}
We start by specifying the following estimate on the discrete gradient $\nabla_{\disc}$.

\begin{lemma} Let $q_h\in X_{\disc_h}$,  then
	\begin{equation}\label{eq:estGrad}
	\norm{\nabla_{\disc_h} q_h}{L^2(\O)} \lesssim \norm{q_h}{1,\disc_h},
	\end{equation}
	where $a\lesssim b$ means that there is a constant $C$, independent of $h$, such that $a\leq Cb$.
\end{lemma}
\begin{proof}
 Using the definition \eqref{eq:stabGrad} of the discrete gradient and the orthogonality of the stabilisation term \eqref{eq:stab_ortho}, we have that 
\begin{equation}\nonumber
\begin{aligned}
\int_\O \nabla_{\disc_h} q_h \cdot \nabla_{\disc_h} q_h \,\,d\x &= \sum_{K\in\mesh_h} \int_K \nabla_{\disc_h} q_h \cdot \nabla_{\disc_h} q_h \,\,d\x\\
&= \sum_{K\in\mesh_h} \int_K \ograd_K q_h \cdot \ograd_K q_h \,\,d\x + \sum_{K\in\mesh_h} \int_K S_k(q_h) \cdot S_K(q_h) \,\,d\x.
\end{aligned}
\end{equation}
We start by looking at the consistent term and use the definition \eqref{eq:discGrad_expression} to obtain
\begin{equation}\nonumber
\begin{aligned} 
\int_\O \ograd_{\disc_h} q_h \cdot \ograd_{\disc_h} q_h \,\,d\x &= \sum_{K\in\mesh_h} \int_K \ograd_K q_h \cdot \ograd_K q_h \,\,d\x\\
&= \sum_{K\in\mesh_h} \sum_{\edge \in \edgescv}|\sigma| (q_\sigma-q_K) \ograd_K q_h\cdot \mathbf{n}_{K,\sigma}.
\end{aligned}
\end{equation}
Applying Cauchy-Schwarz, we then have
\begin{align}
\norm{\ograd_{\disc_h} q_h}{L^2(\O)}^2
 &\leq \sum_{K\in\mesh_h}\bigg(\sum_{\sigma\in\edgescv} \frac{|\sigma|}{d_{K,\sigma}} (q_\sigma-q_K)^2\bigg)^{1/2}\bigg(\sum_{\sigma\in\edgescv} |\sigma| d_{K,\sigma}(\ograd_K q_h\cdot \mathbf{n}_{K,\sigma})^2  \bigg)^{1/2}\nonumber\\
&\leq \bigg(\sum_{K\in\mesh_h}\sum_{\sigma\in\edgescv} \frac{|\sigma|}{d_{K,\sigma}} (q_\sigma-q_K)^2\bigg)^{1/2}\bigg(\sum_{K\in\mesh_h}\sum_{\sigma\in\edgescv} |\sigma| d_{K,\sigma}(\ograd_K q_h \cdot \mathbf{n}_{K,\sigma})^2  \bigg)^{1/2}\nonumber\\
&\leq \norm{q_h}{1,\disc_h}\bigg(\sum_{K\in\mesh_h}\sum_{\sigma\in\edgescv} |\sigma| d_{K,\sigma}(\ograd_K q_h \cdot \mathbf{n}_{K,\sigma})^2  \bigg)^{1/2}\label{ineq:grad}.
\end{align}
Using the fact that 
\[
(\ograd_K q_h \cdot \mathbf{n}_{K,\sigma})^2  \leq \ograd_K q_h \cdot \ograd_K q_h
\]
and
\[\sum_{\edge\in\edgescv}|\sigma|d_{K,\sigma} =d|K|,\] we then have 
\begin{equation}\nonumber
\begin{aligned}
\bigg(\sum_{K\in\mesh_h}\sum_{\sigma\in\edgescv} |\sigma| d_{K,\sigma}(\ograd_K q_h \cdot \mathbf{n}_{K,\sigma})^2  \bigg)^{1/2} &\leq \sqrt{d}\bigg(\sum_{K\in\mesh_h} |K| \ograd_K q_h \cdot \ograd_K q_h \bigg)^{1/2}\\
&= \sqrt{d}\norm{\ograd_{\disc_h} q_h}{L^2(\O)},
\end{aligned}
\end{equation}
which upon substitution to \eqref{ineq:grad}, leads to 
\begin{equation}\label{eq:gradConsistent}
\norm{\ograd_{\disc_h} q_h}{L^2(\O)}^2 \leq d \norm{q_h}{1,\disc_h}^2.
\end{equation}
Now, for the stabilisation term, we have from \eqref{eq:stab_K}, \eqref{eq:stab_Ksig}, and \eqref{eq:area_DKsig} that 
\begin{equation}\nonumber
\begin{aligned}
\sum_{K\in\mesh_h} \int_K S_k(q_h) \cdot S_K(q_h) \,\,d\x &= \sum_{K\in\mesh_h} \sum_{\sigma\in\edges_K} |D_{K,\sigma}| S_{K,\sigma} \cdot S_{K,\sigma}\\
&= d \sum_{K\in\mesh_h} \sum_{\sigma\in\edges_K} \frac{|D_{K,\sigma}|}{d_{K,\sigma}^2} (q_\sigma - q_K - \ograd_{K} q \cdot (\x_\sigma-\x_K))^2\\
&= \sum_{K\in\mesh_h} \sum_{\sigma\in\edges_K} \frac{|\sigma|}{d_{K,\sigma}} (q_\sigma - q_K - \ograd_{K} q \cdot (\x_\sigma-\x_K))^2.\\
\end{aligned}\end{equation}
Using the fact that $(a+b)^2 \leq 2a^2+2b^2$ for any real numbers $a,b$, we then have 
\begin{equation}\nonumber
\begin{aligned}
\sum_{K\in\mesh_h} \int_K S_k(q_h) \cdot S_K(q_h) \,\,d\x &\leq 2  \sum_{K\in\mesh_h} \sum_{\sigma\in\edges_K} \frac{|\sigma|}{d_{K,\sigma}} (q_\sigma - q_K)^2\\
&\,\,+ 2 \sum_{K\in\mesh_h} \sum_{\sigma\in\edges_K} \frac{|\sigma|}{d_{K,\sigma}} (\ograd_{K} q \cdot (\x_\sigma-\x_K))^2.
\end{aligned}
\end{equation}
Now, the first term is simply $\norm{q_h}{1,\disc_h}^2$, so we only need to deal with the second term, which we will denote by $T_2$. Using Cauchy-Schwarz, the regularity of the mesh \ref{assum. mesh_reg}, and \eqref{eq:gradConsistent} we have that 
\begin{equation}\nonumber
\begin{aligned} 
T_2 &\leq 2\sum_{K\in\mesh_h} (\ograd_{K} q \cdot \ograd_{K} q) \sum_{\sigma\in\edges_K} \frac{|\sigma|}{d_{K,\sigma}} \big((\x_\sigma-\x_K) \cdot (\x_\sigma-\x_K)\big) \\
& \lesssim \sum_{K\in\mesh_h} |K|(\ograd_{K} q \cdot \ograd_{K} q)\\
& \lesssim \norm{q_h}{1,\disc_h}^2.
\end{aligned}
\end{equation} 
This allows us to conclude that 
\begin{equation}\label{eq:gradStab}
\sum_{K\in\mesh_h} \int_K S_k(q_h) \cdot S_K(q_h) \,\,d\x\lesssim \norm{q_h}{1,\disc_h}^2.
\end{equation}
Combining the inequalities \eqref{eq:gradConsistent} and \eqref{eq:gradStab} then concludes the proof.
\end{proof}

\subsection{A priori estimates}
\begin{lemma} Let us assume that \ref{assum.Diff}–\ref{assum.Src} hold. Let $\mesh_h$ be an admissible discretization of $\O$ such that
	$\theta > \mathrm{regul}(\mesh_h)$, and let $F_{K,\sigma}^A$ be the advective flux of $q_h \in X_{\disc_h}$ given by \eqref{eq:hybrid_flux} for the velocity
	field $\V \in C^1(\overline{\O})^d$. Then there exists a non-negative
	constant $C_1 \geq 0$ that only depends on $\theta$ and $\V$ such that
	\begin{equation}\label{eq:Lemma-apriori1}
	\begin{aligned}
	&\forall q_h\in X_{\disc_h}: \\
	&\,\,\,\frac{1}{2}\int_\O (\Pi_{\disc_h}q_h)^2 \divg(\V) \,\,d\x\leq \sum_{K\in\mesh_h}\sum_{\sigma\in\edges_K}|\sigma| F_{K,\sigma}^A (q_K-q_\sigma) + C_1 h \norm{q_h}{1,\disc_h}^2.
	\end{aligned}
	\end{equation}
\end{lemma}
\begin{proof}
	
We start by taking note that
\begin{equation}\nonumber
\begin{aligned}
\V_{K,\sigma}^+ &= \frac{1}{2}\V_{K,\sigma}+\frac{1}{2}(\V_{K,\sigma}^++\V_{K,\sigma}^-),\\
\V_{K,\sigma}^- &= -\frac{1}{2}\V_{K,\sigma}+\frac{1}{2}(\V_{K,\sigma}^++\V_{K,\sigma}^-),
\end{aligned}
\end{equation}
so that we can write the fluxes \eqref{eq:hybrid_flux} as 
\begin{equation}\label{eq:advFluxV}
F_{K,\sigma}^A  = \frac{1}{2}(q_K+q_\sigma)\V_{K,\sigma} +\frac{1}{2}(\V_{K,\sigma}^++\V_{K,\sigma}^-)(q_K-q_\sigma)+\gradD q_K \cdot (\x_\sigma-\x_K)\V_{K,\sigma}^+.
\end{equation}

Using \eqref{eq:advFluxV} and the fact that $ \V_{K,\sigma}+\V_{K',\sigma}=0$ for two cells $K$ and $K'$ sharing $\sigma$, we then obtain
	\begin{equation}\nonumber
\begin{aligned}
&\!\!\!\sum_{K\in\mesh_h}\sum_{\sigma\in\edges_K}|\sigma|F_{K,\sigma}^A (q_K-q_\sigma) 
\\
&=
\sum_{K\in\mesh_h}\sum_{\edge\in\edgescv} \frac{|\sigma|}{2}(q_K^2-q_\sigma^2) \V_{K,\sigma} +\sum_{K\in\mesh_h}\sum_{\edge\in\edgescv}\frac{|\sigma|}{2}(\V_{K,\sigma}^++\V_{K,\sigma}^-)(q_K-q_\sigma)^2\\
&\,\,\,+ \sum_{K\in\mesh_h}\sum_{\edge\in\edgescv} |\sigma| \gradD q_K \cdot (\x_\sigma-\x_K)(q_K-q_\sigma)\V_{K,\sigma}^+\\
&= \sum_{K\in\mesh_h}\sum_{\edge\in\edgescv} \frac{|\sigma|}{2}q_K^2 \V_{K,\sigma} +\sum_{K\in\mesh_h}\sum_{\edge\in\edgescv}\frac{|\sigma|}{2}(\V_{K,\sigma}^++\V_{K,\sigma}^-)(q_K-q_\sigma)^2\\
&\,\,\,+ \sum_{K\in\mesh_h}\sum_{\edge\in\edgescv} |\sigma| \gradD q_K \cdot (\x_\sigma-\x_K)(q_K-q_\sigma)\V_{K,\sigma}^+\\ 
&\geq \sum_{K\in\mesh_h}\sum_{\edge\in\edgescv} \frac{|\sigma|}{2}q_K^2 \V_{K,\sigma} + \sum_{K\in\mesh_h}\sum_{\edge\in\edgescv} |\sigma| \gradD q_K \cdot (\x_\sigma-\x_K)(q_K-q_\sigma)\V_{K,\sigma}^+.
\end{aligned} 
\end{equation}
Using the fact that 
\[
\sum_{\edge \in \edgescv} |\sigma| \V_{K,\sigma} = \int_K \divg(\V)\,\,d\x,
\]
we then have
	\begin{equation}\label{eq:ineq1}
	\begin{aligned}
	\frac{1}{2}&\int_\O (\Pi_{\disc_h}q_h)^2  \divg(\V) \,\,d\x \\
	&\!\!\! \leq 
\sum_{K\in\mesh_h}\sum_{\sigma\in\edges_K}|\sigma|F_{K,\sigma}^A (q_K-q_\sigma) + \sum_{K\in\mesh_h}\sum_{\edge\in\edgescv} |\sigma| \gradD q_K \cdot (\x_K-\x_\sigma)(q_K-q_\sigma)\V_{K,\sigma}^+ .
\end{aligned}
\end{equation}
Now, upon applying Cauchy-Schwarz, we obtain
\begin{align*}
&\sum_{K\in\mesh_h}\sum_{\edge\in\edgescv} |\sigma| \gradD q_K \cdot (\x_K-\x_\sigma)(q_K-q_\sigma)\V_{K,\sigma}^+\\
&\,\, \leq \sum_{K\in\mesh_h}\bigg(\sum_{\edge\in\edgescv} |\sigma|\bigg(\gradD q_K \cdot (\x_K-\x_\sigma)\V_{K,\sigma}^+\bigg)^2\bigg)^{1/2} \bigg(\sum_{\edge\in\edgescv} |\sigma|(q_K-q_\sigma)^2\bigg)^{1/2} \\
&\,\, \leq \bigg(\sum_{K\in\mesh_h}\sum_{\edge\in\edgescv} \bigg(|\sigma|\gradD q_K \cdot (\x_K-\x_\sigma)\V_{K,\sigma}^+\bigg)^2\bigg)^{1/2}
 \bigg(\sum_{K\in\mesh_h}\sum_{\edge\in\edgescv}|\sigma|(q_K-q_\sigma)^2\bigg)^{1/2}\\
 &\,\, \lesssim \bigg(\sum_{K\in\mesh_h}\sum_{\edge\in\edgescv} |\sigma|(\gradD q_K \cdot \gradD q_K)   \big((\x_K-\x_\sigma) \cdot (\x_K-\x_\sigma)\big) \bigg)^{1/2} h^{1/2}\norm{(q_h,q_{\edges_h})}{1,D_h,\mathcal{E}_h}\\
 &\,\, \lesssim h\norm{\gradD q_h}{L^2(\O)} \norm{q_h}{1,\disc_h},
\end{align*}
which, together with \eqref{eq:estGrad}, leads to 
\begin{equation}\nonumber
\sum_{K\in\mesh_h}\sum_{\edge\in\edgescv} |\sigma| \gradD q_K \cdot (\x_\sigma-\x_K)(q_K-q_\sigma)\V_{K,\sigma}^+ \lesssim h \norm{q_h}{1,\disc_h}^2.
\end{equation}
Substituting the above inequality into \eqref{eq:ineq1} then leads us to \eqref{eq:Lemma-apriori1}.
\end{proof}

\begin{lemma}\label{lem:est_ch}
	Let us assume that \ref{assum.Diff}-\ref{assum.Src} hold. Let $\mesh_h$ be an admissible discretization of $\O$
	such that $\theta> \mathrm{regul}(D_h)$, and let $F_{K,\sigma}^A$ be the advective flux given by \eqref{eq:hybrid_flux} for
	 $\V \in C^1(\O)^d$. Then, for all solutions
	$c_h$ to the scheme \eqref{eq:scheme_fluxbal}-\eqref{eq:bdcond} with diffusive and advective fluxes defined as in \eqref{eq:diff_fluxes} and \eqref{eq:hybrid_flux}, we have
	\begin{equation}\label{eq:lemmaCoercivity}
	\norm{c_h}{1,\disc_h}^2 \lesssim \norm{f}{L^2(\O)}\norm{c_h}{L^2(\O)}+h\norm{c_h}{1,\disc_h}^2.
	\end{equation} 
\end{lemma}
\begin{proof}
	Take 
	\[
	\int_\O f \Pi_{\disc_h} c_h \,\,d\x= \sum_{K\in\mesh_h} \int_K f c_K \,\,d\x\]
	and let $c_h$ be the solution to the scheme. Then, using the conservation of fluxes, we have
	\begin{equation}\nonumber
	\begin{aligned}
	\int_\O f \Pi_{\disc_h} c_h \,\,d\x&=  \sum_{K\in\mesh_h} \sum_{\edge \in \edgescv} |\sigma|(F_{K,\sigma}^D + F_{K,\sigma}^A) c_K\\
	&=   \sum_{K\in\mesh_h} \sum_{\edge \in \edgescv} |\sigma|F_{K,\sigma}^D(c_K-c_\sigma) + \sum_{K\in\mesh} \sum_{\edge \in \edgescv} |\sigma|F_{K,\sigma}^A(c_K-c_\sigma) .
	\end{aligned}
	\end{equation}
	Applying Cauchy-Schwarz then leads us to 
	\[
	\sum_{K\in\mesh_h} \sum_{\edge \in \edgescv} F_{K,\sigma}^D(c_K-c_\sigma) + \sum_{K\in\mesh_h} \sum_{\edge \in \edgescv} F_{K,\sigma}^A(c_K-c_\sigma)  \leq \norm{f}{L^2(\O)}\norm{c_{h}}{L^2(\O)}.
	\]
Using the definition \eqref{eq:diff_fluxes} for the diffusive fluxes, we have
\[
\sum_{K\in\mesh_h} \sum_{\edge \in \edgescv} |\sigma|F_{K,\sigma}^D(c_K-c_\sigma) = \sum_{K\in\mesh_h} \int_K \Lambda_K \nabla_{\disc_h} c_h \cdot \nabla_{\disc_h} c_h\,\,d\x.
\]
We then use the assumption \ref{assum.Diff} on $\Lambda$ and \cite[Lemma 13.11]{GDMBook16} to establish that
\[ \norm{c_h}{1,\disc_h}^2 \lesssim
\sum_{K\in\mesh_h} \sum_{\edge \in \edgescv} |\sigma|F_{K,\sigma}^D(c_K-c_\sigma).
\]
Combining this with the inequality \eqref{eq:Lemma-apriori1}, and using the fact that $\divg (\V) \geq 0$ then allows us to conclude the proof.
\end{proof}

\subsection{Convergence result}

In this section, we show that the numerical solution of the scheme given by \eqref{eq:scheme_fluxbal}-\eqref{eq:bdcond}, with diffusive and advective fluxes defined as in \eqref{eq:diff_fluxes},\eqref{eq:hybrid_flux} converges to the weak solution of \eqref{eq:model}. For simplicity of exposition, we consider homogeneous Dirichlet boundary conditions. 
\begin{definition}[Weak solution of the advection-diffusion problem \eqref{eq:model}] We say that $c\in H^1(\O)$ is the weak solution of \eqref{eq:model} if for any $\psi \in H^1(\O)$
	\begin{equation}\label{eq:wkForm}
	\int_\O \Lambda \nabla c \cdot \nabla \psi \,\,d\x - \int_\O c\V \cdot \nabla \psi \,\,d\x= \int_\O f \psi \,\,d\x.
	\end{equation}
\end{definition}
\begin{theorem}\label{th:conv}
	Let $c\in H^1(\O)$ be the weak solution to \eqref{eq:model}. Under assumptions \ref{assum.Diff}–\ref{assum.Src},
 let $c_h\in X_{\disc_h}$ be the numerical solution to the scheme \eqref{eq:scheme_fluxbal}-\eqref{eq:bdcond}, with diffusive and advective fluxes constructed as in \eqref{eq:diff_fluxes} and \eqref{eq:hybrid_flux}. Then, for $h \rightarrow 0$, the following hold:
	\begin{enumerate}
		\item $\Pi_{\disc_h} c_h \rightarrow c$ in $L^r(\O)$ for all $r<\frac{2d}{d-2}$ 
		\item $\ograd_{\disc_h} c_h \rightarrow \nabla c$ in $L^2(\O)^d$.
	\end{enumerate}
\end{theorem}
	\begin{proof}
		Using Lemma \ref{lem:est_ch}, we have $\norm{c_h}{1,\disc_h}^2 \lesssim  \norm{f}{L^2(\O)} \norm{\Pi_{\disc_h}c_h}{L^2(\O)}$ when $h$ is small. In view of Lemma \ref{lem:discSob}, we obtain an upper bound on $\norm{c_h}{1,\disc_h}$. Then the result of Lemma \ref{lem:discRel} implies
		the existence of a function $c \in H^1_0 (\O)$ such that, up to a subsequence, $\Pi_{\disc_h}c_h \rightarrow c$ in $L^r(\O)$ for all
		$r < \frac{2d}{d-2}$ and $\ograd_{\disc_h} c_h \rightarrow \nabla c $ weakly in $L^2(\O)^d$. Now, 
		consider $\phy\in C_c^\infty(\O)$ and for $K\in\mesh_h, \sigma \in \edges_h$ write  $\phy_K = \phy(\x_K), \phy_\sigma = \phy(\x_\sigma)$. Let $c_h \in X_{\disc_h}$ be the numerical solution to \eqref{eq:scheme_fluxbal}-\eqref{eq:bdcond}. 
		We multiply the balance of flux equations \eqref{eq:scheme_fluxbal} by $\phy_K$ and take the sum over $K\in\mesh_h$ to obtain
		\begin{equation}\nonumber
		\begin{aligned}
		\sum_{K\in\mesh} \int_K f \phy_K \,\,d\x&= \sum_{K\in\mesh}\sum_{\edge \in \edgescv} |\sigma|F_{K,\sigma} \phy_K. \\
		\end{aligned}
		\end{equation}
		 Using the conservation of internal fluxes, we then have
		 \begin{equation}\nonumber
		 \begin{aligned}
		 	\int_\O f \phy_h	\,\,d\x &= \sum_{K\in\mesh}\sum_{\edge \in \edgescv} |\sigma|F_{K,\sigma} (\phy_K-\phy_\sigma)\\
		 	&= \sum_{K\in\mesh}\sum_{\edge \in \edgescv} |\sigma|F_{K,\sigma}^D (\phy_K-\phy_\sigma)\\
		 	&\,\, +\sum_{K\in\mesh}\sum_{\edge \in \edgescv} |\sigma|\bigg(\big(c_K + \nabla_{\disc_h} c_K \cdot (\x_\sigma - \x_K)\big)\V_{K,\sigma}^+ - c_\sigma \V_{K,\sigma}^- \bigg) (\phy_K-\phy_\sigma)\\
		 		&= \sum_{K\in\mesh}\sum_{\edge \in \edgescv} |\sigma|F_{K,\sigma}^D (\phy_K-\phy_\sigma)+\sum_{K\in\mesh}\sum_{\edge \in \edgescv} |\sigma|c_K \V_{K,\sigma} (\phy_K-\phy_\sigma)\\
		 	&\,\, +\sum_{K\in\mesh}\sum_{\edge \in \edgescv}|\sigma|(c_K- c_\sigma) \V_{K,\sigma}^-(\phy_K-\phy_\sigma)\\
		 	&\,\, + \sum_{K\in\mesh}\sum_{\edge \in \edgescv} |\sigma|\big(\nabla_{\disc_h} c_K \cdot (\x_\sigma - \x_K)\big)\V_{K,\sigma}^+ (\phy_K-\phy_\sigma).\\
		 \end{aligned}
		 \end{equation}
		 We then write the right hand side of the above equation as $T_1+T_2+T_3+T_4$. 
		 For the term $T_1,$ we write a Taylor expansion 
		 \[
		 \phy_\sigma = \phy_K + \nabla \phy(\x_K) \cdot (\x_\sigma - \x_K) + R_{K,\sigma}(\phy),
		 \]
		 where $R_{K,\sigma}(\phy)\lesssim h^2 \norm{\nabla^2 \phy}{\infty}$. This leads to 
		 \begin{equation}\nonumber
		 \begin{aligned}
		 T_1 &= \sum_{K\in\mesh}\sum_{\edge \in \edgescv} |\sigma|F_{K,\sigma}^D \nabla \phy(\x_K) \cdot (\x_K - \x_\sigma) + \sum_{K\in\mesh}\sum_{\edge \in \edgescv} |\sigma|F_{K,\sigma}^D R_{K,\sigma}(\phy)\\
		 &= T_{1,1}+T_{1,2}.
		 \end{aligned}
		 \end{equation}
		 Using the definition \eqref{eq:diff_fluxes} of the diffusive flux and the orthogonality of the stabilisation term \eqref{eq:stab_ortho}, we then get
		 \begin{equation}\nonumber
		 \begin{aligned}
		 T_{1,1} &= \sum_{K\in\mesh_h} \int_K \nabla_{\disc_h} c_h \cdot \nabla \phy (\x_K)\,\,d\x\\
		 &= \int_\O \ograd_{\disc_h} c_h \cdot (\nabla \phy)_h\,\,d\x,
		 \end{aligned}
		 \end{equation}
		 where $(\nabla \phy)_h$ is defined such that $\big((\nabla \phy)_h\big)_{|K} = \phy(\x_K)$ for each $K\in\mesh_h$.
		 Using the weak convergence of $\ograd_{\disc_h} c_h$ and the strong convergence of $(\nabla \phy)_h$, we have that as $h \rightarrow 0$,
		 \begin{equation}\label{eq:convT11}
		 T_{1,1} \rightarrow \int_\O \Lambda \nabla c \cdot \nabla \phy \,\,d\x.
		 \end{equation}
		 Also, as $h \rightarrow 0$, $T_{1,2} \rightarrow 0$.
		 Now, considering $T_2$, we have that 
		 \begin{equation}\nonumber
		 \begin{aligned}
		 T_2 &= \sum_{K\in\mesh} c_K \phy_K \sum_{\sigma\in\edgescv} |\sigma|\V_{K,\sigma} -\sum_{K\in\mesh} c_K \sum_{\edge \in \edgescv} |\sigma|\V_{K,\sigma} \phy_\sigma\\
		 &= \sum_{K\in\mesh} \int_K c_K\phy_K\divg(\V) \,\,d\x  - \sum_{K\in\mesh} c_K \sum_{\edge \in \edgescv} \int_\sigma \phy \V \cdot \mathbf{n}_{K,\sigma} \,\,ds\\
		 &\,\,
		 + \sum_{K\in\mesh} c_K \sum_{\edge \in \edgescv} \int_\sigma(\phy-\phy_\sigma)\V \cdot \mathbf{n}_{K,\sigma}\,\,ds\\
		 &= \int_\O \Pi_{\disc_h}c_h \phy_h \divg(\V) \,\,d\x- \int_\O\Pi_{\disc_h} c_h \divg(\phy \V)\,\,d\x\\
		 &\,\, +\sum_{K\in\mesh} c_K \sum_{\edge \in \edgescv} \int_\sigma(\phy-\phy_\sigma)\V \cdot \mathbf{n}_{K,\sigma}\,\,ds.
		 \end{aligned}
		 \end{equation}
		 Due to the strong convergence of $\Pi_{\disc_h}c_h$, we see that the first two terms on the right hand side converges to 
		 $\int_\O c\phy \divg(\V) \,\,d\x- \int_\O c \divg (\phy V)\,\,d\x$.
		 We now consider the third term.  Due to the fact that $\V_{K,\sigma}$ is conservative, we may write
		 \begin{equation}\label{eq:argT_23}
		 \begin{aligned}
		 \bigg|\sum_{K\in\mesh} c_K \sum_{\edge \in \edgescv} \int_\sigma(\phy-\phy_\sigma)\V \cdot \mathbf{n}_{K,\sigma}\,\,ds\bigg| &= \bigg|\sum_{K\in\mesh} \sum_{\edge \in \edgescv} (c_K-c_\sigma) \int_\sigma(\phy-\phy_\sigma)\V \cdot \mathbf{n}_{K,\sigma}\,\,ds\bigg|\\
		 &\lesssim h \norm{\nabla \phy}{\infty} \sum_{K\in\mesh} \sum_{\edge \in \edgescv} |\sigma| |c_K-c_\sigma|.
		 \end{aligned}
		 \end{equation}
		 Using Cauchy-Schwarz and the boundedness of $\norm{c_h}{1,\disc_h}$, we have
		 \begin{equation}\nonumber
		 \sum_{K\in\mesh} \sum_{\edge \in \edgescv} |\sigma||c_K-c_\sigma| \leq (d|\O|)^{1/2} \norm{c_h}{1,\disc_h} \lesssim 1.
		 \end{equation}
		 Thus, as $h\rightarrow 0$, we have
		 \begin{equation}\label{eq:convT2}
		 T_2 \rightarrow  \int_\O c\phy \divg(\V)\,\,d\x - \int_\O c \divg (\phy \V)\,\,d\x.
		 \end{equation}
		 For $T_3$, we use the boundedness $|\V_{K,\sigma}| \lesssim 1$ and an argument similar to \eqref{eq:argT_23} in order to establish that as $h\rightarrow 0$,
		 \begin{equation}\label{eq:convT3}
		 T_3 \rightarrow 0.
		 \end{equation}
		 Finally, we consider the term $T_4$.
		 Here, we have
		 \begin{equation}\nonumber
		 \begin{aligned}
		 T_4 \lesssim  \norm{\nabla \phy}{\infty} \sum_{K\in\mesh}\sum_{\edge \in \edgescv} |\sigma|\big| \nabla_{\disc_h} c_K \cdot (\x_\sigma - \x_K)\big|,
		 \end{aligned}
		 \end{equation}
		 By Cauchy-Schwarz, we then have 
		 \begin{equation}\nonumber
		 \begin{aligned}
		 &\sum_{K\in\mesh}\sum_{\edge \in \edgescv} |\sigma|\big| \nabla_{\disc_h} c_K \cdot (\x_\sigma - \x_K)\big|\\
		  &\,\,\,\leq \sum_{K\in\mesh}(\sum_{\edge\in\edgescv} |\sigma||\x_\sigma-\x_K| \nabla_{\disc_h} c_K \cdot \nabla_{\disc_h} c_K)^{1/2}  (\sum_{\edge\in\edgescv} |\sigma| |\x_\sigma-\x_K|)^{1/2} \\
		 &\,\,\,\lesssim h \sum_{K\in\mesh} \int_K \nabla_{\disc_h} c_K \cdot \nabla_{\disc_h} c_K\,\,d\x,
		 \end{aligned}
		 \end{equation}
		 which, together with \eqref{eq:estGrad}, leads to 
		 \begin{equation}\nonumber
		 \begin{aligned}
		 T_4 \lesssim h \norm{\nabla \phy}{\infty} \norm{c_h}{1,\disc_h}.
		 \end{aligned}
		 \end{equation}
		 Hence, as $h\rightarrow 0$,
		 \begin{equation}\label{eq:convT4}
		 T_4 \rightarrow 0.
		 \end{equation}
		 Combining the results in \eqref{eq:convT11}, \eqref{eq:convT2}, \eqref{eq:convT3}, \eqref{eq:convT4} then shows us that indeed the numerical solution converges to the weak solution \eqref{eq:wkForm} of the advection-diffusion problem.
		 
		 Following Step 3 of the proof in \cite[Theorem 3.7]{VDM11-adv-diff}, we can establish that the weak convergence $\ograd_{\disc_h} c_h \rightarrow \nabla c$ is, in fact, strong.
	\end{proof}

We note here that key to the convergence analysis is the property \eqref{eq:estGrad} of the gradient used for the linear term in \eqref{eq:hybrid_flux}. Hence, flux-limited second-order upwind fluxes, such as those in \cite{BM-2dFV}, are also covered by the analysis presented above. In particular, we note that these fluxes take the form
\begin{equation} \label{eq:flux_limited_2dUp}
F_{K,\sigma}^A = \bigg(c_K + \phi_K(c)\widetilde{\nabla}_{\disc} c_K\cdot (\x-\x_K)\bigg)\V_{K,\sigma}^+ - c_\sigma \V_{K,\sigma}^-,
\end{equation}
where $\phi_K: X_\disc \rightarrow [0,1]$.  Since $\phi$ is bounded between $0$ and $1$, $\phi \widetilde{\nabla}_\disc q$ satisfies the estimate \eqref{eq:estGrad} for any $q\in X_\disc$; thus, the convergence results still hold when the advective fluxes \eqref{eq:hybrid_flux} are replaced with \eqref{eq:flux_limited_2dUp}.

\section{Numerical tests} \label{sec:Numtests}
In this section, we present numerical tests of the fully local second-order upwind scheme defined by \eqref{eq:scheme_fluxbal}-\eqref{eq:bdcond} over the domain $\O$, with diffusive and advective fluxes defined as in \eqref{eq:diff_fluxes} and \eqref{eq:hybrid_flux}. This will be compared with the hybridised upwind scheme (with advective fluxes as in \eqref{eq:hybrid_firstOrder}), and the cell-centered second-order upwind scheme (with advective fluxes \eqref{eq:2up_gen}). We note however that for cell-centered schemes, the advective fluxes \eqref{eq:2up_gen} cannot be straightforwardly computed near the boundary of the domain; hence, we switch into a first-order upwind scheme near the boundary of the domain. This will lead to a system with $N_K$ + $N_{e_{\mathrm{ext}}}$ equations and unknowns, where $N_{e_{\mathrm{ext}}}$ is the number of boundary edges. 

\subsection{1D test: $\epsilon$-sensitivity} \label{sec:test_eps}
We start by performing a test over a one dimensional domain $\O=(0,1)$. Here, we check for the $\epsilon$-sensitivity of the schemes. This test checks that the numerical diffusion introduced by the scheme (if any) is not too much. In particular, if the actual solution of the problem contains a thin boundary layer, we also expect the numerical solution to observe the same property. Consider the ODE
\[
c'(x)-\epsilon c''(x) = 0 \quad \mbox{for} \quad x\in(0,1),
\]
with Dirichlet boundary conditions
\[
c(0)=1, c(1)=0.
\]
The solution for this differential equation can be calculated exactly, and is given by 
\[
c(x) = \frac{1}{e^{-1/\epsilon}-1}\bigg(e^{1/\epsilon(x-1)}-1\bigg).
\]

Here, the exact solution $c$ has a boundary layer, which is controlled by the diffusion parameter $\epsilon$. A numerical scheme with a good $\epsilon$-sensitivity should allow us to capture the boundary layer, even when it can only be resolved up to 1 grid cell on the mesh. For this test, we consider a mesh with 100 equidistant cells. Hence, the size of each cell is given to be $h=0.01$. We then consider the diffusion parameters $\epsilon=2^{-4},2^{-6},2^{-8},2^{-10}$. In Figures \ref{fig.2d_eps_sol_up}-\ref{fig.2d_eps_sol_2dhyb}, left, the numerical solutions (dashed lines) are plotted against the exact solution (solid lines) for different values of $\epsilon$. We then plot on the right of Figures \ref{fig.2d_eps_sol_up}-\ref{fig.2d_eps_sol_2dhyb}, the pointwise error values $c(x)-\Pi_{\disc}c(x)$. Since the numerical solutions only deviate from the actual solution near the boundary layer, we zoom in and present the plots for $x\in[0.5,1]$.

\begin{figure}[h]
	\begin{tabular}{cc}
		\includegraphics[width=0.5\linewidth]{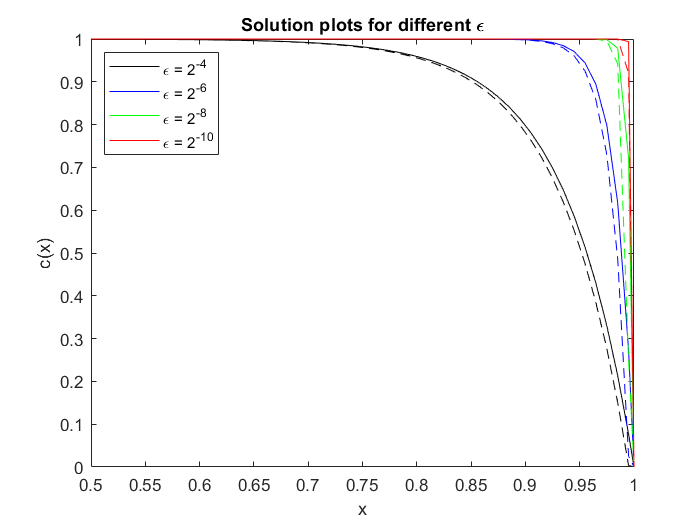} & \includegraphics[width=0.5\linewidth]{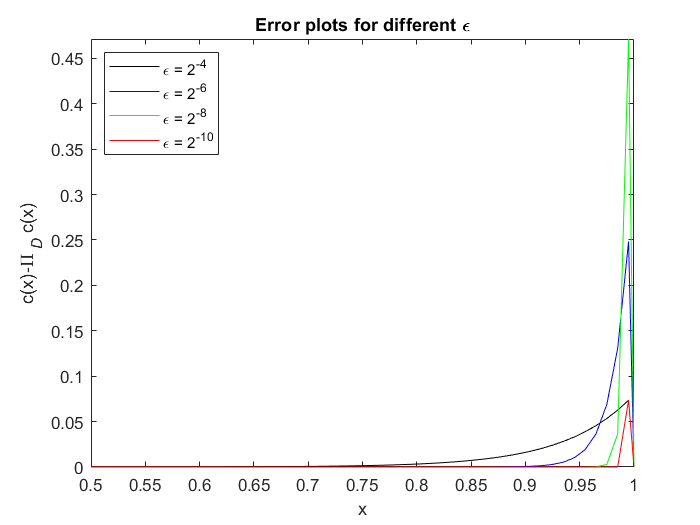} 
	\end{tabular}
	\caption{ Exact solution against numerical solution, hybridised upwind scheme. (left:  solution plots, right: errors).}
	\label{fig.2d_eps_sol_up}
\end{figure}

\begin{figure}[h]
	\begin{tabular}{cc}
		\includegraphics[width=0.5\linewidth]{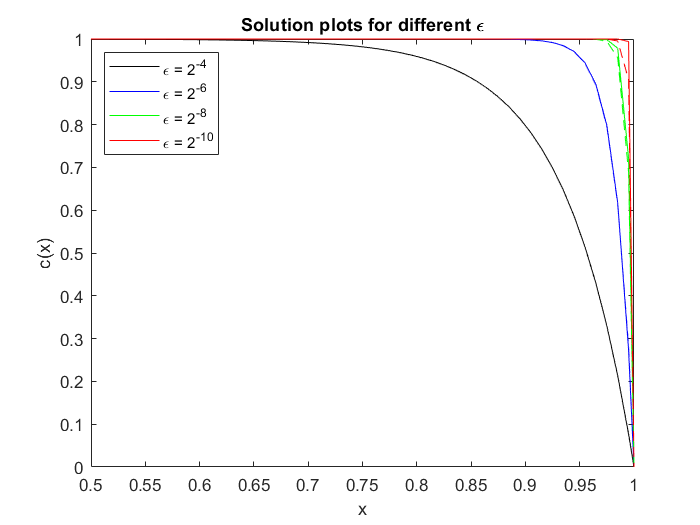} & \includegraphics[width=0.5\linewidth]{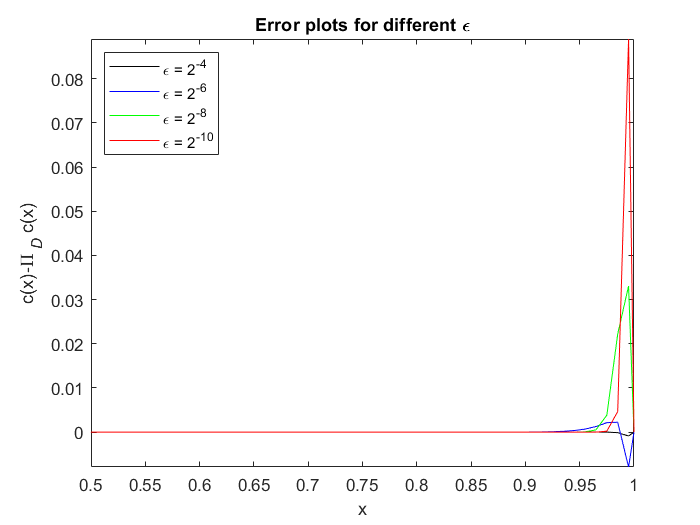} 
	\end{tabular}
	\caption{ Exact solution against numerical solution, cell-centered second-order scheme. (left:  solution plots, right: errors).}
	\label{fig.2d_eps_sol_2dup}
\end{figure}

\begin{figure}[h]
	\begin{tabular}{cc}
		\includegraphics[width=0.5\linewidth]{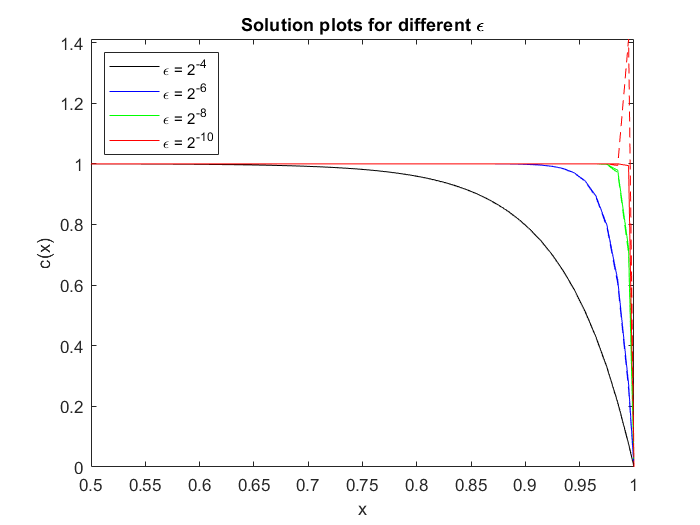} & \includegraphics[width=0.5\linewidth]{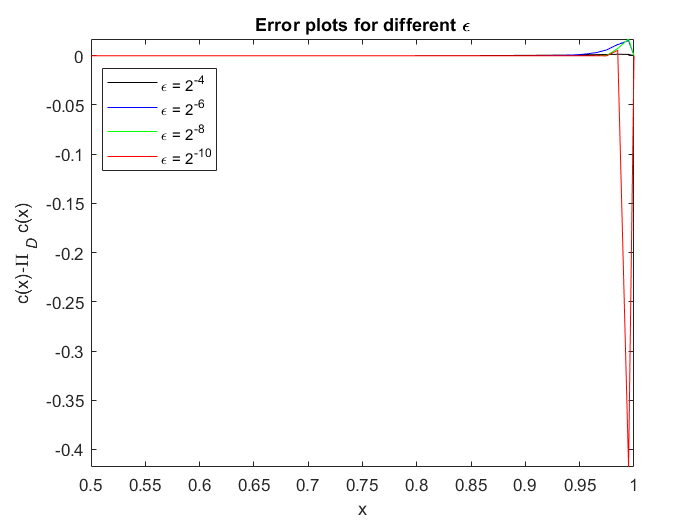} 
	\end{tabular}
	\caption{ Exact solution against numerical solution, hybridised second-order scheme. (left:  solution plots, right: errors).}
	\label{fig.2d_eps_sol_2dhyb}
\end{figure}

We start by looking at the first-order upwind scheme. We note in Figure \ref{fig.2d_eps_sol_up}, right that the quantity $c(x)-\Pi_{\disc}c(x)$ is always nonnegative. This is expected from first-order upwind schemes due to the numerical diffusion it introduces, which leads to the smoothening of the solution and widening of the boundary layer. This is illustrated in Figure \ref{fig.2d_eps_sol_up}, left. We now move on to the cell-centered second-order scheme. Here, we see a better agreement between the numerical solution and the exact solution in the interior of the domain. However, since we switch to a first-order upwind scheme near the boundary of the domain, we see that the solution and error plots in Figure \ref{fig.2d_eps_sol_2dup} are similar to those in Figure \ref{fig.2d_eps_sol_up}. Now, upon looking at the hybridised second-order scheme in Figure \ref{fig.2d_eps_sol_2dhyb}, we see that for $\epsilon \geq 2^{-8}$, the numerical solutions obtained via the hybridised second-order scheme behave in a manner that is very similar to the exact solution. We note however, that for $\epsilon=2^{-10}$, the numerical solution obtained from the hybridised second-order scheme exhibits an overshoot. This is expected, since second-order linear schemes do not guarantee monotonicity and stability of the numerical solution. One way to resolve this is to do, as with the cell-centered scheme, switching into a first-order upwind scheme near the boundary of the domain. However, this will only result to numerical solutions similar to that in Figure \ref{fig.2d_eps_sol_2dup}, and will not be helpful in resolving the boundary layer. Another way to resolve the overshoot is by introducing an artificial vanishing diffusion term. In order to do so, we consider, for the hybridised second-order scheme, a diffusion parameter $\epsilon+h^{1.5}$ instead. This leads to a significant improvement in the numerical results, as can be seen in Figure \ref{fig.2d_eps_sol_2dhyb_vd}. In particular, we are now able to capture the boundary layer, with only small overshoots (less than 5\%) in the numerical solution. 

\begin{figure}[h]
	\begin{tabular}{cc}
		\includegraphics[width=0.5\linewidth]{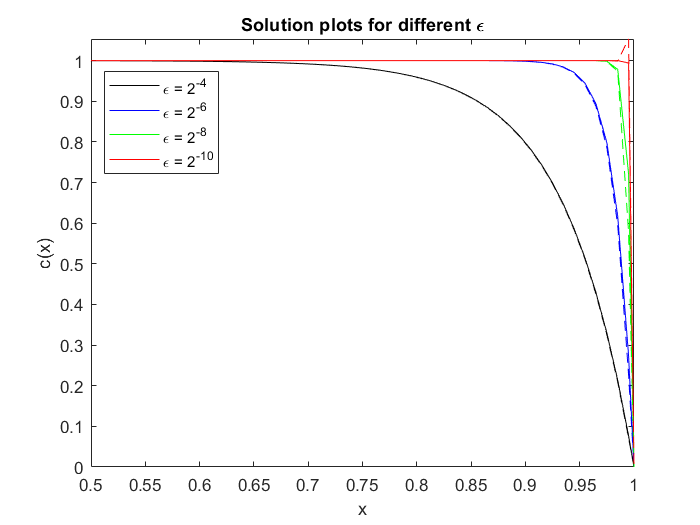} & \includegraphics[width=0.5\linewidth]{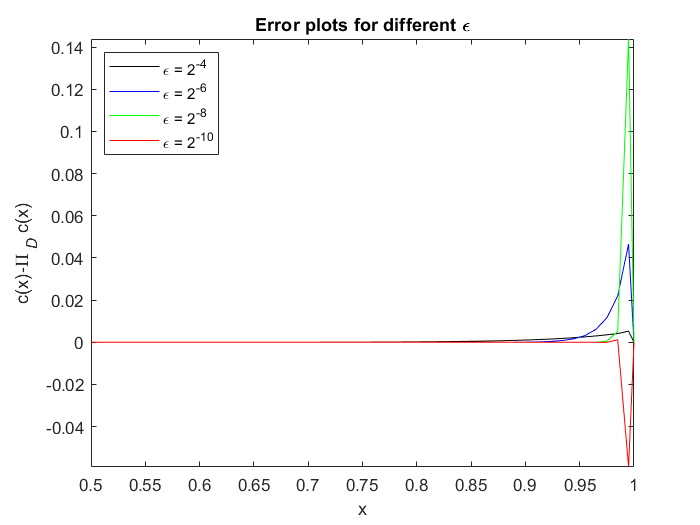} 
	\end{tabular}
	\caption{ Exact solution against numerical solution, hybridised second-order scheme with vanishing diffusion. (left:  solution plots, right: errors).}
	\label{fig.2d_eps_sol_2dhyb_vd}
\end{figure}

To summarise, this test shows us that the numerical diffusion introduced by the hybridised first-order upwind scheme smoothens the solution, leading to a widening of the boundary layer. On the other hand, with the introduction of vanishing diffusion, the numerical solution provided by the hybridised second-order scheme is able to capture the boundary layer properly, with only minimal overshoots.

\subsection{2D tests}
We now proceed to tests in 2D, which will be done on the domain $\O=(0,1)\times(0,1)$. The numerical tests will be done on different mesh types, starting with regular Cartesian and triangular meshes (see Figure \ref{fig:Meshes_1}), followed by distorted meshes, which include moved Cartesian, moved triangular, and Kershaw type meshes \cite{HH08,K98-Kershaw_mesh} (see Figure \ref{fig:Meshes_2}). These mesh types will be denoted by $\mesh_1,\mesh_2,\dots \mesh_5$, respectively. Here, the moved Cartesian and moved triangular meshes in Figure \ref{fig:Meshes_2} are constructed following the guidelines provided in \cite{L10-monotoneFV}. That is, starting with a uniform Cartesian and triangular mesh as in Figure \ref{fig:Meshes_1}, if the maximum diameter of the cells are given by $h$, then the internal nodes $(x,y)$ are perturbed randomly by taking
\[
\hat{x} := x + 0.4\beta_x h, \quad \hat{y} := y +0.4\beta_y h,
\]
where $\beta_x,\beta_y$ are random values between $-0.5$ and $0.5$. 
\begin{figure}[h!]
	\begin{tabular}{cc}
		\includegraphics[width=0.4\linewidth]{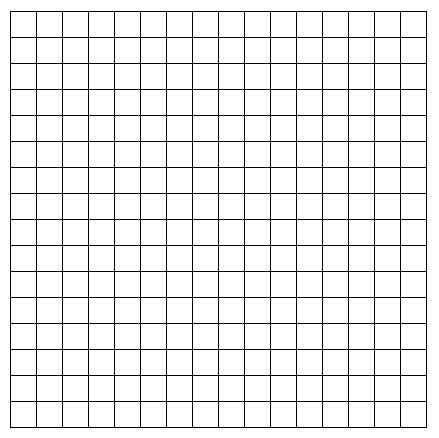} & 	\quad 	\includegraphics[width=0.4\linewidth]{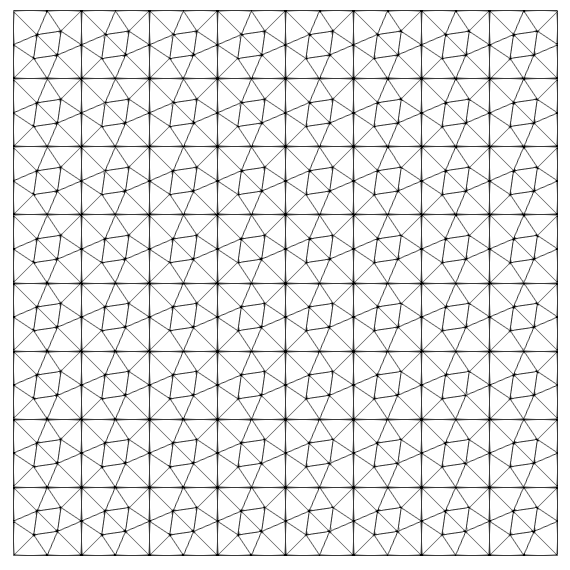}
		
	\end{tabular}
	\caption{Mesh types: $\mesh_1$ Cartesian (left); $\mesh_2$ triangular (right). }
	\label{fig:Meshes_1}
\end{figure}
\begin{figure}[h!]
	\begin{tabular}{ccc}
		\includegraphics[width=0.3\linewidth]{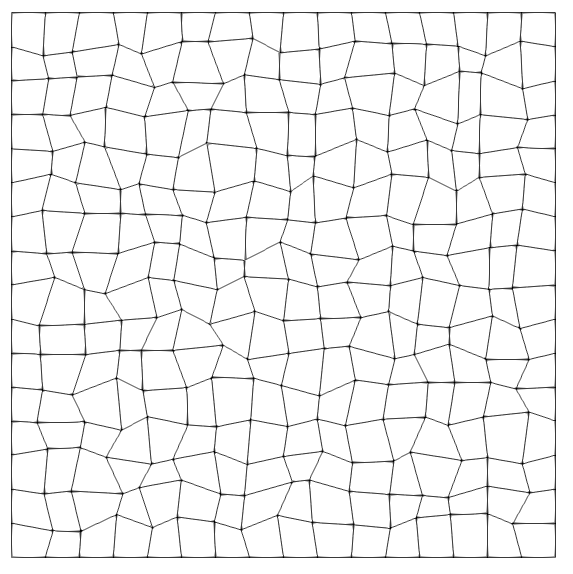} & 	\quad 	\includegraphics[width=0.3\linewidth]{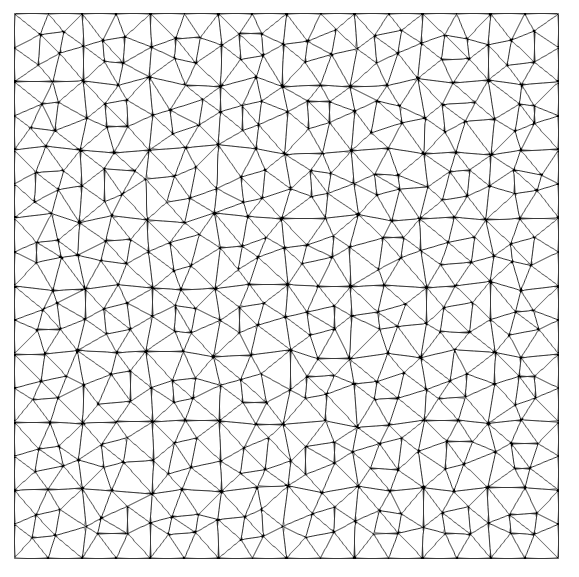} &\quad 
		\includegraphics[width=0.3\linewidth]{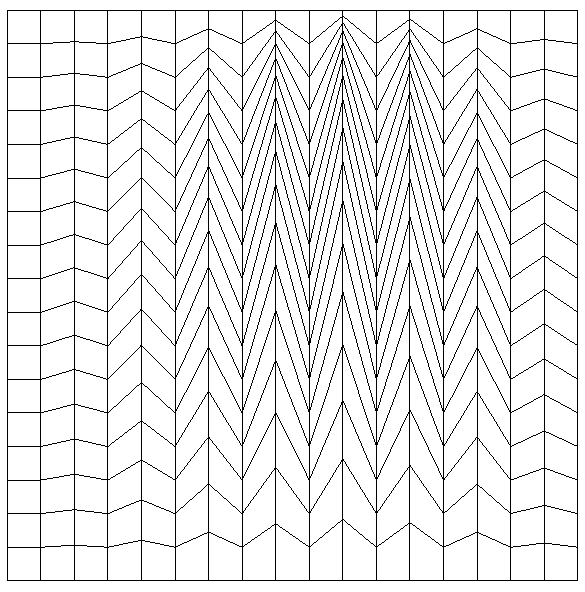}
	\end{tabular}
	\caption{Mesh types: $\mesh_3$  moved Cartesian (left); $\mesh_4$  moved triangular (middle); $\mesh_5$  Kershaw (right). }
	\label{fig:Meshes_2}
\end{figure}

For the convergence tests, we measure the relative solution error
\begin{equation}\nonumber 
E_c:= \dfrac{\norm{\Pi_{\disc_h}c_h - c}{L^2(\O)}}{\norm{c}{L^2(\O)}},
\end{equation}
and the relative error in the discrete gradient
\begin{equation}\nonumber 
E_g := \dfrac{\norm{\nabla_{\disc_h} c_h - \nabla c}{L^2(\O)}}{\norm{c}{H^1(\O)}}.
\end{equation}
In order to have a more detailed comparison, we also give the number of DOFs required to implement each of the schemes for certain mesh types.
\begin{table}[h!]
\caption{Mesh size $h$ and DOFs for hybridised and cell-centered schemes for each refinement level $r$, mesh $\mesh_1$.}\label{tab:DOFs}
\begin{tabular}{cccccc}
	$r$ & $h$ & $N_K$ & $N_e$ & DOFs(hybridised) & DOFs(cell-centered)\\
	\hline
	1&3.535e-01 & 16 & 40 & 56 & 32 \\
	2&1.767e-01 & 64& 144 & 208 & 96 \\
	3&8.838e-02 & 256 & 544 & 800 & 320\\
	4&4.419e-02 & 1024 & 2112 & 3136 & 1152\\
	5&2.209e-02 & 4096& 8320& 12416& 4352\\
	6&1.104e-02 & 16384& 33024 & 49408 & 16896 \\
\end{tabular}

\end{table}

\subsubsection{Convergence test, smooth solution} \label{sec:test_smooth}
We start with performing a convergence test for the advection-diffusion equation \eqref{eq:model} with prescribed solution 
\[c(x,y) = \sin(\pi x) \sin(\pi y).\] Here, we set the velocity field $\V = [1,2]$ and in order to have an anisotropic advection-dominated problem, we set the diffusion tensor 
\[
\Lambda = \begin{bmatrix}
1.5\times10^{-4} & 10^{-6}\\
10^{-6}& 10^{-8} 
\end{bmatrix}.
\]

\begin{figure}[h]
	\begin{tabular}{cc}
\begin{tikzpicture}[scale=0.75]
\begin{loglogaxis}[
xlabel=$h$,
ylabel=$E_c$
]

\axispath\draw
(-0.25,-6.0)
--  (-2.25,-10.0)  node[pos=0.6,anchor=south east] {2}
--  (-0.25,-10.0)
-- cycle;
\draw
(-0.25,-8.0)
--  (-2.25,-10.0)  
--  (-0.25,-10.0)
-- cycle;
\addplot[mark=o, color=blue] plot coordinates {
	(3.536e-01,   5.192e-02)
	(1.768e-01,   1.287e-02)
	(8.838e-02,   3.208e-03)
	(4.419e-02,   7.997e-04)
	(2.209e-02,  1.989e-04)
	(1.105e-02,  4.934e-05)
};

\addplot[mark=square, color=blue] plot coordinates {
	(8.838e-02,   1.592e-01)
	(4.419e-02,   8.402e-02)
	(2.209e-02,  4.318e-02)
	(1.105e-02,  2.178e-02)
};

\addplot[mark=diamond, color=blue] plot coordinates {
	(3.536e-01,   1.956e-01)
	(1.768e-01,   6.991e-02)
	(8.838e-02,   2.251e-02)
	(4.419e-02,   6.339e-03)
	(2.209e-02,  1.847e-03)
	(1.105e-02,  5.903e-04)
};

\addplot[mark=o, color=red] plot coordinates {
	(7.243e-01,   5.344e-02)
(3.642e-01,   1.895e-02)
(2.130e-01,   5.641e-03)
(1.002e-01,  1.519e-03)
(5.382e-02,  3.915e-04)
(2.693e-02,  9.901e-05)
};

\addplot[mark=square, color=red] plot coordinates {
	(2.130e-01,   1.584e-01)
	(1.002e-01,  8.424e-02)
	(5.382e-02,  4.321e-02)
	(2.693e-02,  2.177e-02)
};
\addplot[mark=diamond, color=red] plot coordinates {
	(7.243e-01,   1.644e-01)
	(3.642e-01,   7.114e-02)
	(2.130e-01,   2.754e-02)
	(1.002e-01,  1.053e-02)
	(5.382e-02,  4.778e-03)
	(2.693e-02,  2.129e-03)
};

\end{loglogaxis}
\end{tikzpicture}
&
\begin{tikzpicture}[scale=0.75]
\begin{loglogaxis}[
xlabel=$h$,
ylabel=$E_g$,
]

\axispath
\draw
(-0.85,-9.5)
--  (-2.85,-11.5)  
        node[pos=0.5,anchor=south east] {1}
--  (-0.85,-11.5)
-- cycle;
\draw
(-0.85,-10.5)
--  (-2.85,-11.5)  
--  (-0.85,-11.5)
-- cycle;
\addplot[mark=o, color=blue] plot coordinates {
	(3.536e-01,   2.691e-04)
	(1.768e-01,   1.354e-04)
	(8.838e-02,   7.340e-05)
	(4.419e-02,   4.245e-05)
	(2.209e-02,  2.601e-05)
	(1.105e-02,  1.652e-05)
};

\addplot[mark=square, color=blue] plot coordinates {
	(8.838e-02,   2.586e-01)
	(4.419e-02,   1.676e-01)
	(2.209e-02,  1.111e-01)
	(1.105e-02,  7.535e-02)
};

\addplot[mark=diamond, color=blue] plot coordinates {
	(1.768e-01,   2.524e-01)
	(8.838e-02,   1.520e-01)
	(4.419e-02,   9.862e-02)
	(2.209e-02,  6.675e-02)
	(1.105e-02,  4.606e-02)
};

\addplot[mark=o, color=red] plot coordinates {
	(2.130e-01,   4.270e-02)
	(1.002e-01,  2.235e-02)
	(5.382e-02,  9.725e-03)
	(2.693e-02,  4.322e-03)
};

\addplot[mark=square, color=red] plot coordinates {
	(2.130e-01,   3.130e-01)
	(1.002e-01,  2.484e-01)
	(5.382e-02,  2.270e-01)
	(2.693e-02,  2.190e-01)
};
\addplot[mark=diamond, color=red] plot coordinates {
	(2.130e-01,   2.655e-01)
	(1.002e-01,  2.130e-01)
	(5.382e-02,  1.972e-01)
	(2.693e-02, 1.844e-01)
};
\end{loglogaxis}
\end{tikzpicture}
\end{tabular}
	\caption{Convergence plot, test \ref{sec:test_smooth}. $\circ-$hybridised second-order, $\square-$hybridised first-order, $\diamond-$cell-centered second-order,  blue: $\mesh_1$, red: $\mesh_3$. The slopes of the triangles are order $h^2$ and $h$ on the left, $h$ and $h^{0.5}$ on the right.} \label{fig.conv_plots1}
\end{figure}
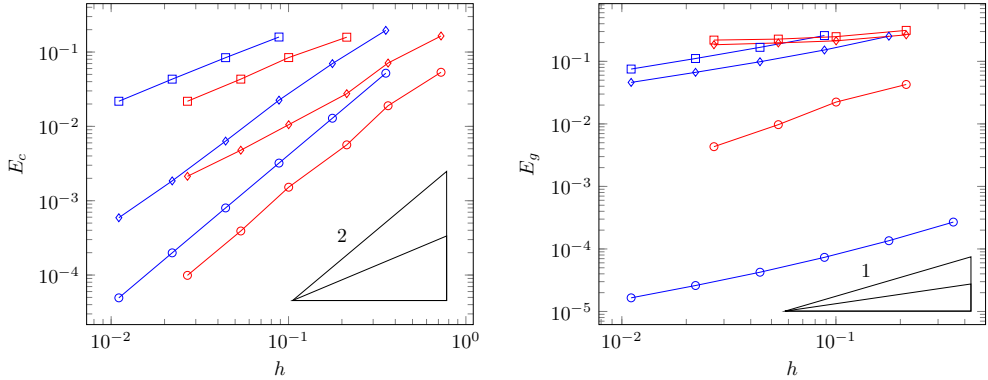

As can be seen in Figures \ref{fig.conv_plots1} and \ref{fig.conv_plots1_2}, the hybridised second-order scheme attains second-order convergence in the solution and first-order convergence in the gradient, except on Cartesian type meshes, where the gradient converges at a rate of approximately $h^{0.7}$. This is due to the fact that the error in the gradient is much smaller than the error in the solution profile; once the error in the solution is smaller than that of the gradient, we expect to observe first-order convergence for the gradient. Firstly, we note that the hybridised second-order scheme is an improvement over the hybridised (first-order) upwind scheme, which has a solution profile that converges with $O(h)$. Moreover, for all mesh types, the solution of the hybridised second-order scheme on the second refinement level is already much better than that of the hybridised upwind scheme on the finest mesh. 

We now look at the second-order cell-centered scheme. We see that in general, this gives an improvement over that of the hybridised first-order scheme, i.e. the second-order cell-centered scheme provides a more accurate solution with the same mesh size $h$. However, no significant improvement is observed on the accuracy in terms of the gradient reconstruction. Upon comparing the cell-centered scheme with the hybridised second-order scheme, we see that the hybridised scheme performs better, both in terms of mesh size $h$ and in terms of the number of DOFs needed to solve the system. In particular, for mesh type $\mesh_1$, the solution of the second-order cell-centered scheme on the finest mesh, which requires solving a system with 16896 unknowns, gives a relative error of 5.903e-04. However, the hybridised second-order scheme already achieves a relative error of 7.997e-04 by solving a system of only 3136 unknowns. A similar observation in terms of the advantages (in DOFs and mesh size) can also be made on the other meshes. We also note here that although the discrete gradient for the cell-centered second-order scheme converges with an order $h^{0.5}$ on $\mesh_1$, no convergence is observed in the other meshes. This can be explained by the fact that for advection-dominated problems, the advective fluxes are the dominant factors, and hence, the conservation of fluxes \eqref{eq:fluxcons} approximately imposes that
\[
F_{K,\sigma}^A + F_{L,\sigma}^A = 0.
\]
This leads to taking $c_\sigma$ from the upwind direction, thus making the discrete gradient \eqref{eq:discGrad_expression} equivalent to \eqref{eq:discGrad_upwind}, which also explains why the error in the gradient $E_g$ for both the hybridised upwind scheme and the cell-centered second-order scheme are quite close to each other.

\begin{figure}[h]
	\begin{tabular}{cc}
		\begin{tikzpicture}[scale=0.75]
		\begin{loglogaxis}[
		xlabel=$h$,
		ylabel=$E_c$
		]
		
		\axispath\draw
		(-0.75,-5.0)
		--  (-2.75,-9.0)  node[pos=0.6,anchor=south east] {2}
		--  (-0.75,-9.0)
		-- cycle;
		\draw
		(-0.75,-7.0)
		--  (-2.75,-9.0)  
		--  (-0.75,-9.0)
		-- cycle;
		\addplot[mark=o, color=green] plot coordinates {
			(2.500e-01,   4.139e-02)
			(1.250e-01,   9.195e-03)
			(6.250e-02,   2.198e-03)
			(3.125e-02,   5.355e-04)
			(1.562e-02,  1.268e-04)
		};

		\addplot[mark=square, color=green] plot coordinates {
			(2.500e-01,   2.393e-01)
			(1.250e-01,   1.268e-01)
			(6.250e-02,   6.445e-02)
			(3.125e-02,   3.246e-02)
			(1.562e-02,  1.614e-02)
		};
		
		\addplot[mark=diamond, color=green] plot coordinates {
			(2.500e-01,   1.230e-01)
			(1.250e-01,   5.620e-02)
			(6.250e-02,   2.905e-02)
			(3.125e-02,   1.468e-02)
			(1.562e-02,  7.336e-03)
		};
		
		\addplot[mark=o, color=brown] plot coordinates {
			(4.610e-01,   3.627e-02)
			(2.463e-01,   1.108e-02)
			(1.329e-01,   3.016e-03)
			(6.571e-02,  7.155e-04)
			(3.402e-02, 1.872e-04)
		};
		
		\addplot[mark=square, color=brown] plot coordinates {
			(4.610e-01,   2.863e-01)
			(2.463e-01,   1.373e-01)
			(1.329e-01,   6.955e-02)
			(6.571e-02,  3.543e-02)
			(3.402e-02,  1.771e-02)
		};
		
		\addplot[mark=diamond, color=brown] plot coordinates {
			(4.610e-01,   1.468e-01)
			(2.463e-01,   6.183e-02)
			(1.329e-01,   3.138e-02)
			(6.571e-02,  1.561e-02)
			(3.402e-02,  7.950e-03)
		};
		
		\addplot[mark=o, color=black] plot coordinates {
			(3.287e-01,   5.103e-02)
			(1.665e-01,  1.218e-02)
			(1.115e-01,  5.243e-03)
			(8.385e-02,  2.907e-03)
		};
		\addplot[mark=square, color=black] plot coordinates {
			(3.287e-01,   3.280e-01)
			(1.665e-01,  2.011e-01)
			(1.115e-01,  1.451e-01)
			(8.385e-02,  1.131e-01)
		};
		
		\addplot[mark=diamond, color=black] plot coordinates {
			(3.287e-01,   1.430e-01)
			(1.665e-01,  6.313e-02)
			(1.115e-01,  3.548e-02)
			(8.385e-02,  2.219e-02)
		};
		
		\end{loglogaxis}
		\end{tikzpicture}
		&
		\begin{tikzpicture}[scale=0.75]
		\begin{loglogaxis}[
		xlabel=$h$,
		ylabel=$E_g$,
		]
		
		\axispath
		\draw
		(-0.75,-3.5)
		--  (-2.25,-5.0)  node[pos=0.5,anchor=south east] {1}
		--  (-0.75,-5.0)
		-- cycle;
		\draw
		(-0.75,-4.25)
		--  (-2.25,-5.0)  
		--  (-0.75,-5.0)
		-- cycle;
		\addplot[mark=o, color=green] plot coordinates {
			(2.500e-01,   1.571e-01)
			(1.250e-01,   6.125e-02)
			(6.250e-02,   2.802e-02)
			(3.125e-02,   1.351e-02)
			(1.562e-02,  6.330e-03)
		};

		\addplot[mark=square, color=green] plot coordinates {
			(2.500e-01,   4.272e-01)
			(1.250e-01,   3.207e-01)
			(6.250e-02,   2.742e-01)
			(3.125e-02,   2.547e-01)
			(1.562e-02,  2.457e-01)
		};
		
		\addplot[mark=diamond, color=green] plot coordinates {
			(2.500e-01,   3.678e-01)
			(1.250e-01,   2.700e-01)
			(6.250e-02,   2.346e-01)
			(3.125e-02,   2.139e-01)
			(1.562e-02,  2.039e-01)
		};
		
		\addplot[mark=o, color=brown] plot coordinates {
			(4.610e-01,   1.122e-01)
			(2.463e-01,   6.846e-02)
			(1.329e-01,   3.227e-02)
			(6.571e-02,  1.490e-02)
			(3.402e-02,  7.305e-03)
		};
		
		\addplot[mark=square, color=brown] plot coordinates {
			(4.610e-01,   4.598e-01)
			(2.463e-01,   3.746e-01)
			(1.329e-01,   3.419e-01)
			(6.571e-02,  3.316e-01)
			(3.402e-02,  3.189e-01)
		};
		
		\addplot[mark=diamond, color=brown] plot coordinates {
			(4.610e-01,   3.849e-01)
			(2.463e-01,   3.141e-01)
			(1.329e-01,   2.892e-01)
			(6.571e-02,  2.714e-01)
			(3.402e-02,  2.614e-01)
		};
		
		\addplot[mark=o, color=black] plot coordinates {
			(3.287e-01,   1.216e-01)
			(1.665e-01,  3.396e-02)
			(1.115e-01,  1.616e-02)
			(8.385e-02,  9.811e-03)
		};
		\addplot[mark=square, color=black] plot coordinates {
			(3.287e-01,   9.770e-01)
			(1.665e-01,  7.565e-01)
			(1.115e-01,  6.320e-01)
			(8.385e-02,  5.515e-01)
		};
		\addplot[mark=diamond, color=black] plot coordinates {
			(3.287e-01,   9.599e-01)
			(1.665e-01,  6.902e-01)
			(1.115e-01,  5.528e-01)
			(8.385e-02,  4.686e-01)
		};
		\end{loglogaxis}
		\end{tikzpicture}
	\end{tabular}
	\caption{Convergence plot, test \ref{sec:test_smooth}. $\circ-$hybridised second-order, $\square-$hybridised first-order, $\diamond-$cell-centered second-order, green: $\mesh_2$, brown: $\mesh_4$, black: $\mesh_5$. The slopes of the triangles are order $h^2$ and $h$ on the left, $h$ and $h^{0.5}$ on the right.} \label{fig.conv_plots1_2}
\end{figure}

Upon having a closer look at Figures \ref{fig.conv_plots1} and \ref{fig.conv_plots1_2}, we observe that the second-order cell-centered scheme was able to attain second-order convergence on the regular Cartesian mesh $\mesh_1$, whilst only first-order convergence on the other types of meshes. This can be explained more clearly by looking at the gradient in Figures \ref{fig.conv_plots1} and \ref{fig.conv_plots1_2}, right. Here, we see that the approximate gradient \eqref{eq:discGrad_upwind} converges for Cartesian type meshes, whereas it does not converge on the other types of meshes. Hence, the linear term $\widetilde{\nabla}_\disc c_K \cdot (\x-\x_K)$ added onto the advective flux \eqref{eq:hybrid_flux} is not accurate and thus does not help improve the convergence of the solution. This means that an improvement over the formulation $\eqref{eq:discGrad_upwind}$ of the discrete gradient $\widetilde{\nabla}_\disc c_K$, such as that proposed in \cite{BJ89-upwind-generic_mesh},  would be needed in order to apply a second-order cell-centered scheme on generic meshes.  This is not straightforward to implement; however, the observations made on the Cartesian meshes $\mesh_1$ are sufficient to support the claim that hybridised second-order schemes perform better than cell-centered schemes.

\subsubsection{Convergence test, solution with boundary layers} \label{sec:test_bdLayer}
We now perform a test for a solution with a boundary layer. For this test case, we prescribe an exact solution 
\[
c(x,y) = \bigg(x-e^{\frac{2(x-1)}{\nu}}\bigg)\bigg(y^2-e^{\frac{3(y-1)}{\nu}}\bigg).
\]
Here, we set the velocity field $\V=[2,3]$ and take $\Lambda = \nu \mathbf{I}$, where $\mathbf{I}$ is the identity matrix. Setting $\nu=10^{-4}$ leads to an advection-dominated problem for which the solution is characterized by a boundary layer near the top and right side of the domain. As in \cite{VDM11-adv-diff,L10-monotoneFV,MR08-FV_advection_diffusion}, the goal of our numerical tests is to demonstrate that the scheme has good convergence properties and produces numerical solutions without oscillations in a subdomain outside the boundary layer. Hence, we measure only the errors in the subdomain $[0,0.8] \times [0,0.8]$. 

\begin{figure}[h]
	\begin{tabular}{cc}
	\begin{tikzpicture}[scale=0.75]
	\begin{loglogaxis}[
	xlabel=$h$,
	ylabel=$E_c$,
	]
	
	\axispath\draw
	(-1.00,-7.5)
	--  (-3.00,-11.5)  node[pos=0.6,anchor=south east] {2}
	--  (-1.00,-11.5)
	-- cycle;
	\draw
	(-1.00,-9.5)
	--  (-3.00,-11.5)  
	--  (-1.00,-11.5)
	-- cycle;
	\addplot[mark = o, color = blue] plot coordinates {
		(8.838e-02,   8.734e-04)
		(4.419e-02,   2.174e-04)
		(2.209e-02,  5.184e-05)
		(1.105e-02,  1.274e-05)
	};
	
	\addplot[mark = o, color = green] plot coordinates {
		(1.250e-01,    2.274e-03)
		(6.250e-02,   2.968e-04)
		(3.125e-02,   7.207e-05)
		(1.562e-02,  1.802e-05)
	};
	
	\addplot[mark = o, color = red] plot coordinates {

		(2.130e-01,   1.559e-03)
		(1.002e-01,  3.826e-04)
		(5.382e-02,  1.045e-04)
		(2.693e-02,  2.862e-05)
	};
	
	\addplot[mark = o , color = brown] plot coordinates {
		(2.463e-01,    2.1835e-03)
		(1.329e-01,   4.1546e-04)
		(6.571e-02,   1.0834e-04)
		(3.402e-02,  3.035e-05)
	};
	
	\addplot[mark = o, color = black] plot coordinates {
		(3.287e-01,     1.760e-02)
		(1.665e-01,    4.231e-03)
		(1.115e-01,   1.870e-03)
		(8.385e-02,   1.037e-03)
	};
	\end{loglogaxis}
	\end{tikzpicture} &
	
		\begin{tikzpicture}[scale=0.75]
	\begin{loglogaxis}[
	xlabel=$h$,
	ylabel=$E_g$
	]
	
	\axispath\draw
	(-1.00,-9.5)
	--  (-3.00,-11.5)  node[pos=0.5,anchor=south east] {1}
	--  (-1.00,-11.5)
	-- cycle;
	\draw
	(-1.00,-10.5)
	--  (-3.00,-11.5)  
	--  (-1.00,-11.5)
	-- cycle;
	\addplot[mark = o, color = blue] plot coordinates {
		(8.838e-02,   7.237e-04)
		(4.419e-02,   1.808e-04)
		(2.209e-02,  4.404e-05)
		(1.105e-02,  1.083e-05)
	};
	
	\addplot[mark = o, color = green] plot coordinates {
		(1.250e-01,    3.312e-02)
		(6.250e-02,   6.341e-03)
		(3.125e-02,   3.171e-03)
		(1.562e-02,  1.584e-03)
	};
	
	\addplot[mark = o, color = red] plot coordinates {
		
		(2.130e-01,   1.184e-02)
		(1.002e-01,  7.363e-03)
		(5.382e-02,  3.929e-03)
		(2.693e-02,  2.058e-03)
	};
	
	\addplot[mark = o, color = brown] plot coordinates {
		(2.463e-01,    4.095e-02)
		(1.329e-01,   7.055e-03)
		(6.571e-02,   3.550e-03)
		(3.402e-02,  1.851e-03)
	};
	
	\addplot[mark = o, color = black] plot coordinates {
		(3.287e-01,     6.521e-02)
		(1.665e-01,    5.244e-03)
		(1.115e-01,   2.277e-03)
		(8.385e-02,   1.233e-03)
	};

	\end{loglogaxis}
	\end{tikzpicture}
\end{tabular}
	\caption{Convergence plot for the hybridised second-order scheme on different mesh types (blue: $\mesh_1$, green: $\mesh_2$, red: $\mesh_3$, brown: $\mesh_4$, black: $\mesh_5$), test \ref{sec:test_bdLayer}. The slopes of the triangles are order $h^2$ and $h$ on the left, $h$ and $h^{0.5}$ on the right.} \label{fig.conv_plots2}
\end{figure}
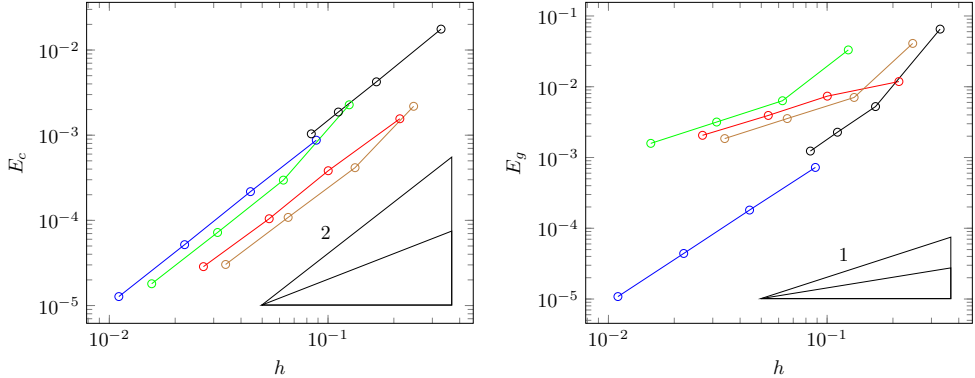

As can be seen in Figure \ref{fig.conv_plots2}, the hybridised second-order scheme provides numerical solutions that are second-order convergent, and gradients that are first-order convergent, regardless of the mesh. Comparison with the hybridised upwind scheme and the cell-centered second-order scheme yielded similar results as test \ref{sec:test_smooth}, and are no longer presented here.

We now study the shock-capturing behavior of the hybridised second-order scheme by plotting the numerical solution in Figure \ref{fig.oscillating_sol}. We only plot on mesh types $\mesh_1$ and $\mesh_3$ (one regular, one irregular), as the numerical solution on other types of meshes exhibit a similar behavior. 

\begin{figure}[h]
	\begin{tabular}{cc}
		\includegraphics[width=0.45\linewidth]{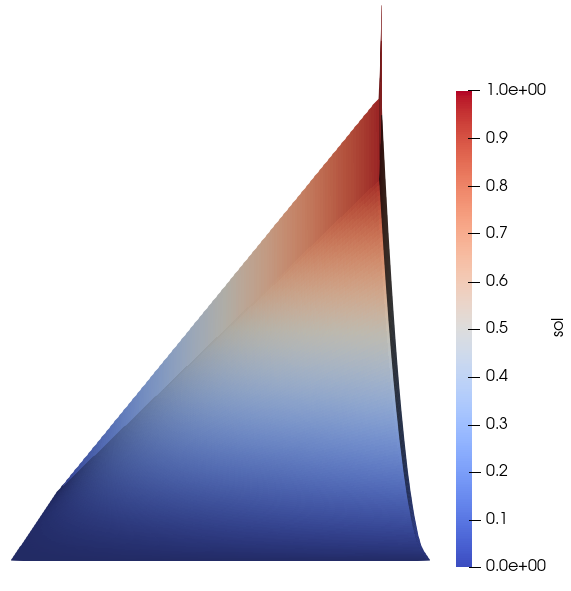} & \includegraphics[width=0.45\linewidth]{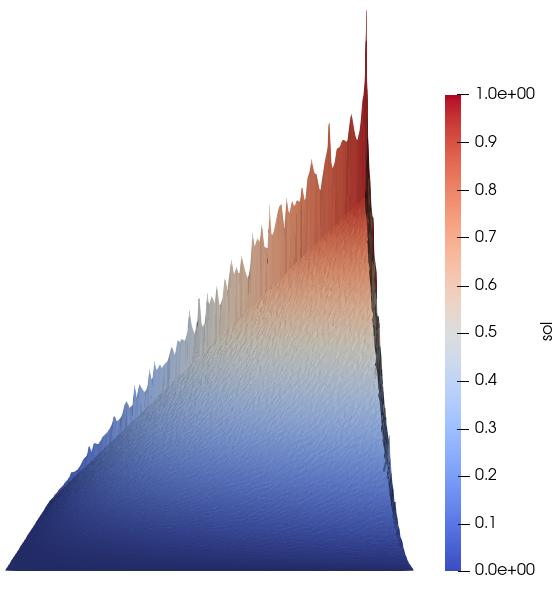} 
	\end{tabular}
	\caption{Solution profile, hybridised second-order scheme, test \ref{sec:test_bdLayer}. (left:  $\mesh_1$, right:  $\mesh_3$).}
	\label{fig.oscillating_sol}
\end{figure}

As can be seen, non-physical oscillations develop at the boundary of the domain; this is expected for second-order schemes which are not flux-limited or total variation diminishing. Moreover, these oscillations are worse on distorted meshes than on regular meshes. In comparison, the hybridised first-order upwind scheme provides solutions which are bounded between 0 and 1, without any non-physical oscillations (see Figure \ref{fig.upwind_sol}). 

\begin{figure}[h]
	\begin{tabular}{cc}
		\includegraphics[width=0.45\linewidth]{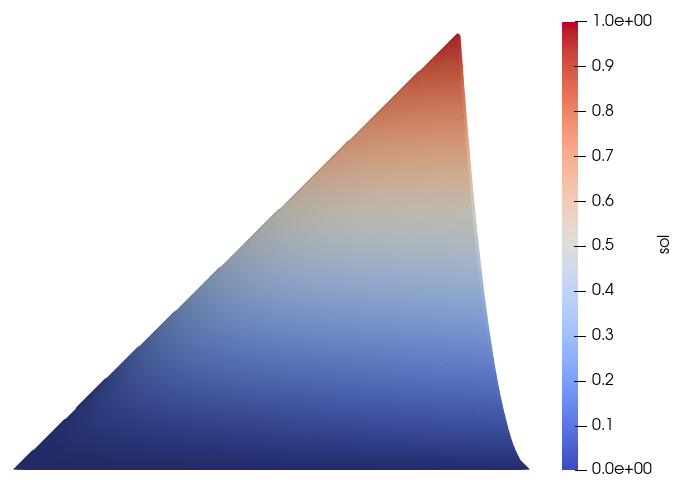} & \includegraphics[width=0.45\linewidth]{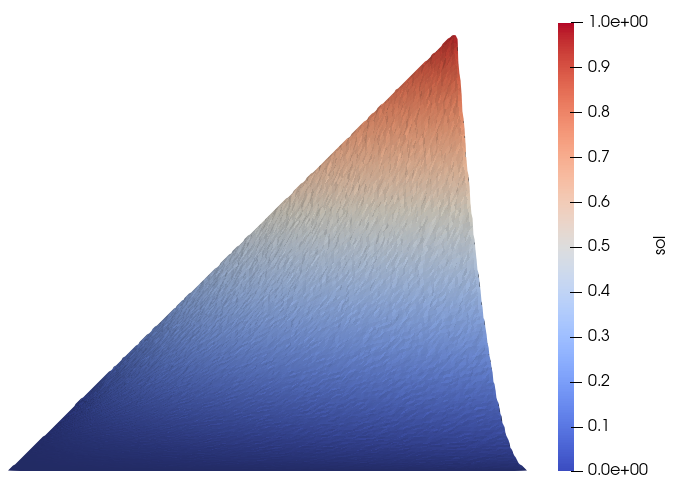} 
	\end{tabular}
	\caption{Solution profile, hybridised upwind scheme, test \ref{sec:test_bdLayer}. (left:  $\mesh_1$, right:  $\mesh_3$).}
	\label{fig.upwind_sol}
\end{figure}

In order to mitigate the non-physical oscillations for the hybridised second-order schemes, we want, as in test \ref{sec:test_eps}, to introduce an artificial vanishing diffusion term. To this end, employ an idea that is similar to that in \cite{DEPT19-vanishing_diff}. That is, if $\V_K$ is an approximation of $\V$ at cell $K$ and if $\Lambda_K$ is an approximation of $\Lambda$ at cell $K$ with diagonalisation $\Lambda_K = U_K' D_K U_K$, we consider 
\begin{equation}\label{eq:mod_diff}
\widetilde{\Lambda}_K = U_K' (D_K+|\V_K|h^{1.5}) U_K,
\end{equation}
and use $\widetilde{\Lambda}_K$ instead in the definition of the diffusive fluxes \eqref{eq:diff_fluxes}.  As can be seen in Figure \ref{fig.vd_sol}, introducing a vanishing diffusion term and using the modified diffusion tensor \eqref{eq:mod_diff}  allows us to mitigate the non-physical oscillations. Moreover, the solutions here are now bounded between 0 and 1. We note, however, that since the artificial diffusion vanishes at a rate of $h^{1.5}$, our scheme reduces to order 1.5; this is not the optimal order 2 convergence, but still offers an improvement over the hybridised first-order scheme. Alternatively, the use of nonlinear flux-limited schemes, as in \cite{BM-2dFV,MR08-FV_advection_diffusion} can also mitigate the non-physical oscillations, whilst preserving second-order accuracy. 

\begin{figure}[h]
	\begin{tabular}{cc}
		\includegraphics[width=0.45\linewidth]{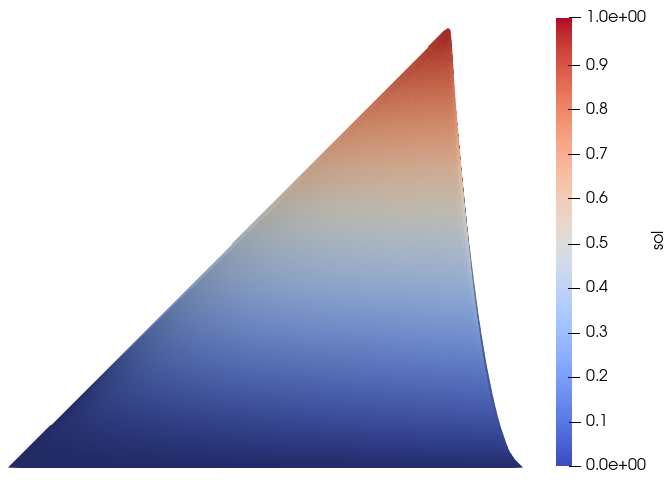} & \includegraphics[width=0.45\linewidth]{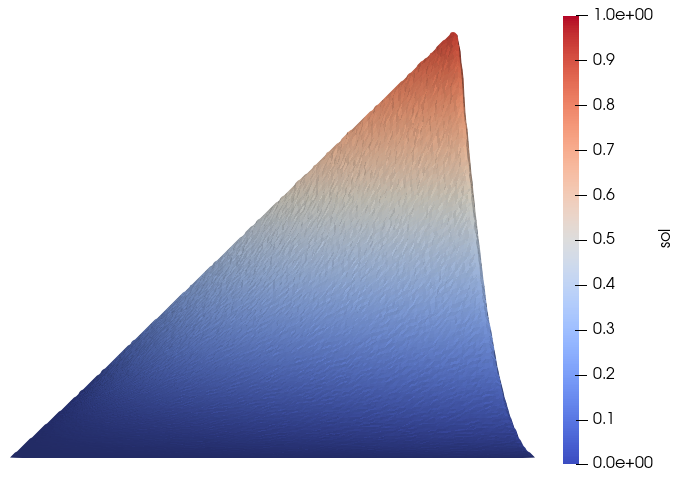} 
	\end{tabular}
	\caption{Solution profile, hybridised second-order scheme with vanishing diffusion, test \ref{sec:test_bdLayer}. (left: $\mesh_1$, right: $\mesh_3$).}
	\label{fig.vd_sol}
\end{figure}


\subsubsection{Strongly anisotropic heterogeneous and convection-dominated case}\label{sec:hetero_test}
Finally, we present a numerical test which involves a strongly heterogeneous and anisotropic diffusion tensor, as described in \cite{VDM11-adv-diff,ESZ09-ADG}. For this test, an exact analytic solution is not available, so we comment on the qualitative properties of the numerical solution. Here, homogeneous Dirichlet boundary conditions are imposed, and we use a source term $f(x,y)= 10^{-2} \exp(-(r-0.35)^2/0.005)$, where $r^2 = (x-0.5)^2+(y-0.5)^2$. The diffusion tensor is piecewise constant, defined in the following subdomains: $\O_1=(0,2/3) \times (0,2/3), \O_2 = (2/3,1)\times (0,2/3), \O_3 = (2/3,1) \times (2/3,1), \O_4 = (0,2/3)\times(2/3,1)$, with
\[
\Lambda = 
\begin{bmatrix}
10^{-6} & 0 \\
0 & 1
\end{bmatrix} \qquad \mbox{ in } \O_1 \mbox{ and } \O_3,
\]
and
\[\Lambda = 
\begin{bmatrix}
1 & 0 \\
0 & 10^{-6}
\end{bmatrix} \qquad \mbox{ in } \O_2 \mbox{ and } \O_4.
\]
The velocity field considered is $\V=(40x(2y-1)(x-1),-40y(2x-1)(y-1))^T$, which simulates a counterclockwise rotation. Figure \ref{fig.hetero_sol} shows the numerical solution obtained from the hybridised second-order scheme on a regular and distorted mesh ($\mesh_1$ and $\mesh_3$), respectively. Here, the distorted mesh $\mesh_3$ is modified so that it matches the discontinuities (see Figure \ref{fig.modified_mesh}). The numerical solutions obtained on other types of meshes exhibit similar behaviors.
\begin{figure}[h]

	\begin{tabular}{cc}
		\includegraphics[width=0.42\linewidth]{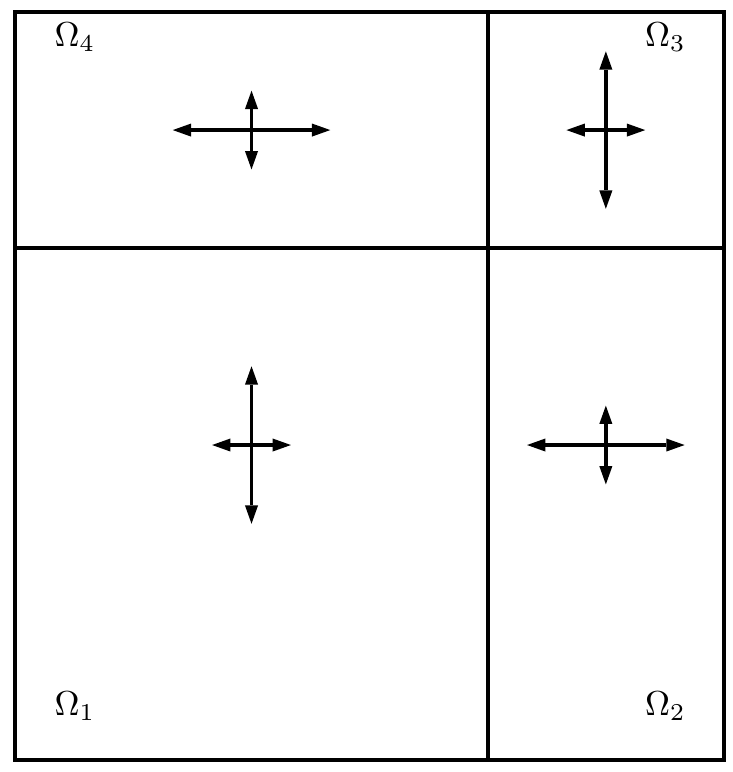} & \includegraphics[width=0.45\linewidth]{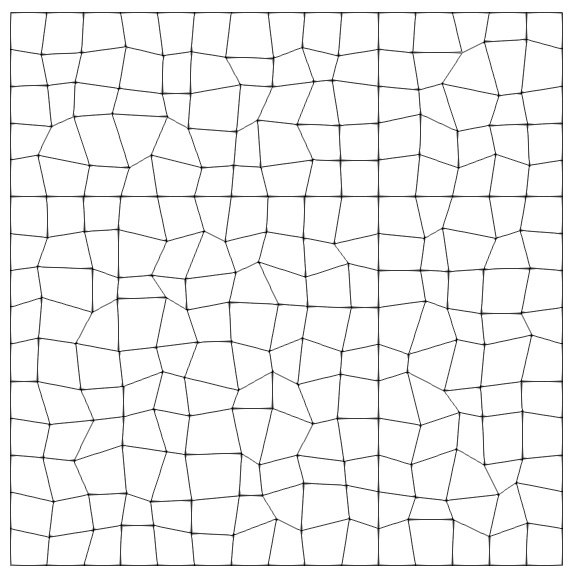} 
	\end{tabular}
		\caption{Diffusion tensor and mesh, test \ref{sec:hetero_test}. }
	\label{fig.modified_mesh}
\end{figure}
\begin{figure}[h]
	\begin{tabular}{cc}
		\includegraphics[width=0.45\linewidth]{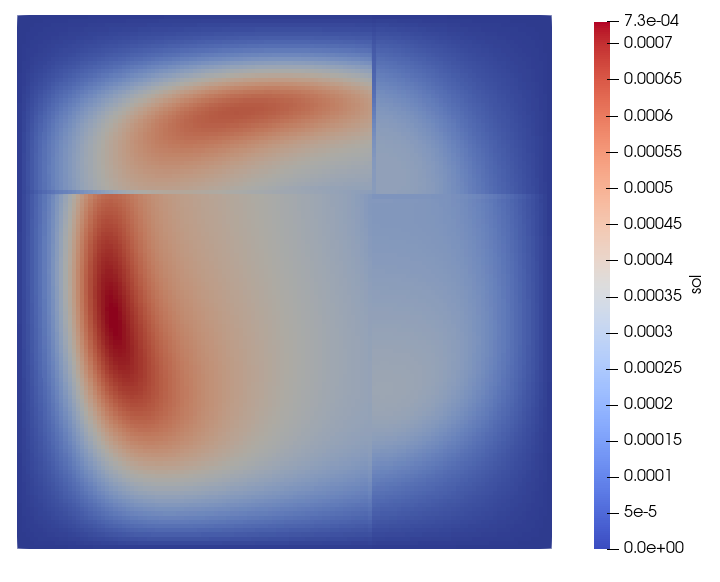} & \includegraphics[width=0.45\linewidth]{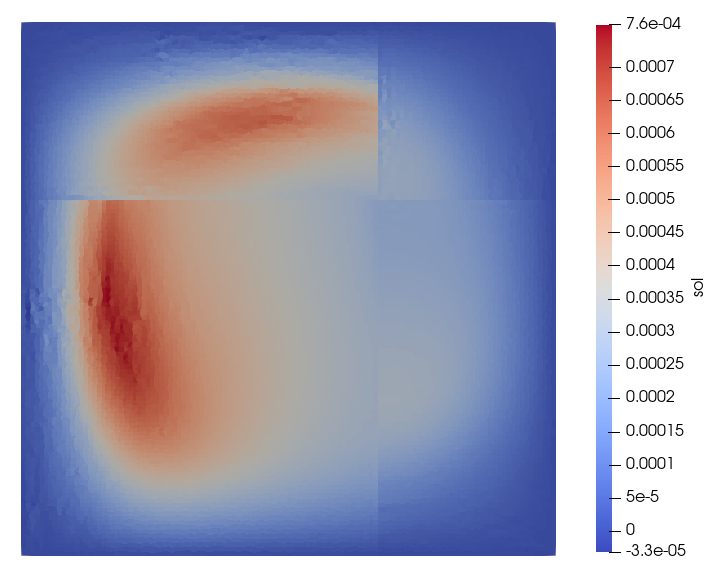} 
	\end{tabular}
	\caption{Solution profile, test \ref{sec:hetero_test}. (left: $\mesh_1$, right: $\mesh_3$).}
	\label{fig.hetero_sol}
\end{figure}

Here, we observe maximum values of $7.3 \times 10^{-4}$ and $7.6 \times 10^{-4}$ for the regular and distorted meshes, respectively. The scheme works well on regular meshes in the sense that the solution profile is similar to those observed in \cite{VDM11-adv-diff,ESZ09-ADG}. Also, the maximum value of $7.3 \times 10^{-4}$ is very close to $6.9\times 10^{-4}$ in the literature. However, it can be seen in Figure \ref{fig.hetero_sol}, right, that some spurious oscillations are present on the distorted mesh. As with test \ref{sec:test_bdLayer}, such a problem is not encountered when using a hybridised upwind scheme. These oscillations can be mitigated by either refining the mesh, or using the modified diffusion tensor \eqref{eq:mod_diff}, as seen in Figure \ref{fig.hetero_sol1}. The main advantage, however, of using the modified diffusion tensor over mesh refinement is that a better quality of the solution profile is obtained without having to introduce additional DOFs for solving the system.

\begin{figure}[h]
	\begin{tabular}{cc}
		\includegraphics[width=0.45\linewidth]{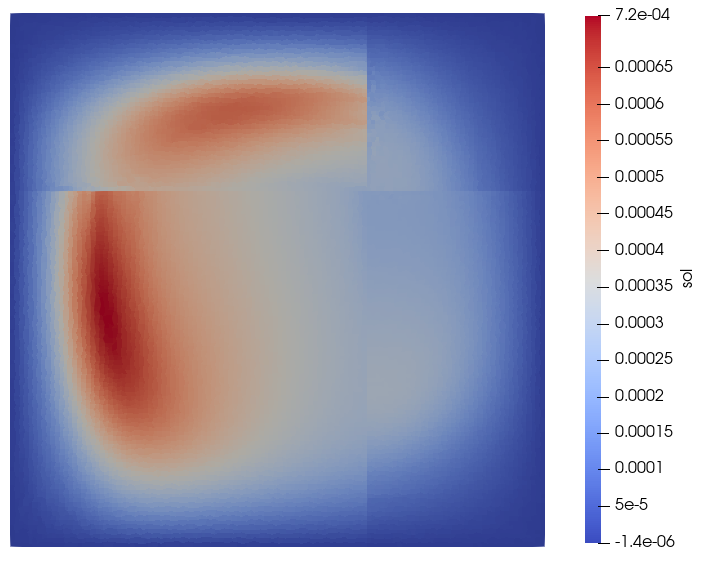} & \includegraphics[width=0.45\linewidth]{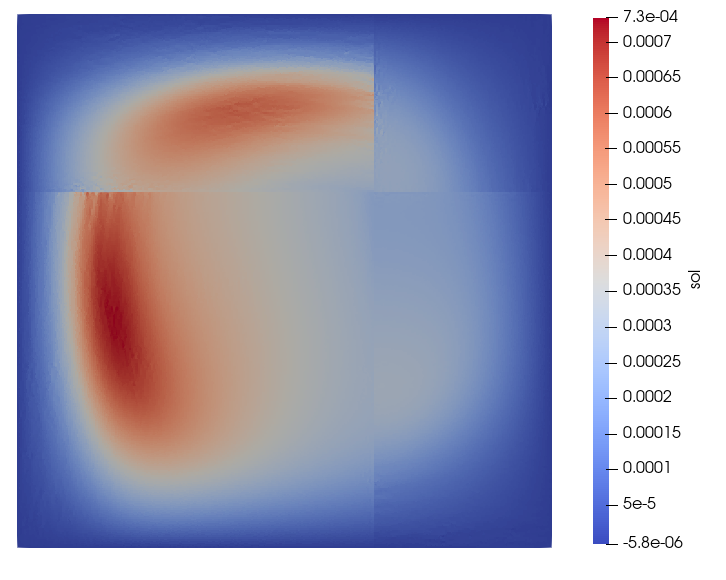} 
	\end{tabular}
	\caption{Solution profile, test \ref{sec:hetero_test}, $\mesh_3$. (left: vanishing diffusion, right: mesh refinement).}
	\label{fig.hetero_sol1}
\end{figure}
\section{Conclusion}
In this work, we proposed a fully local hybridised second-order finite volume scheme for advection-diffusion equations. We then presented a convergence proof for these hybridised second-order schemes, which also cover flux-limited variants. Numerical results were then presented to compare the hybridised second-order scheme with the hybridised upwind scheme, and the classical cell-centered second-order scheme. Firstly, we note that the hybridised second-order scheme provided solutions which are second-order convergent, and gradients which are first-order convergent. It was also shown that the hybridised second-order scheme had a good $\epsilon$-sensitivity. Upon comparison with the cell-centered scheme, the hybridised scheme can achieve the same level or even better accuracy on much coarser meshes. Moreover, the hybridised scheme can straightforwardly be extended onto generic meshes, whereas further improvement on the approximate gradient \eqref{eq:discGrad_upwind} is needed in order to apply the cell-centered schemes onto generic meshes, otherwise the convergence reduces to first order. Another advantage of the hybridised scheme over the cell-centered second-order schemes is that the stencil only depends on local values, and does not need information from neighboring cells, which allows it to be straightforwardly implemented near the boundaries of the domain. Moreover, static condensation can be employed in order to implement the hybridised scheme efficiently. Comparison with hybridised upwind schemes also shows the advantage of the hybridised second-order schemes in terms of the convergence of the solution and the gradient. However, in some instances, for solutions with boundary layers, the solution from the hybridised second-order scheme exhibited non-physical oscillations. This is expected from linear second-order schemes. In comparison, the hybridised upwind scheme does not encounter such problems. In order to mitigate these non-physical oscillations, we propose an idea which involves introducing artificial vanishing diffusion to the hybridised second-order scheme. This reduces the order of convergence to 1.5, but still offers an improvement over the hybridised upwind scheme. One prospect for future work would involve trying to determine when artificial diffusion is needed, and locate on which cells it needs to be introduced, so that the second-order accuracy is preserved over regions for which the solution is smooth. Another avenue for future work would involve extending these ideas to hybrid high order schemes \cite{DSGD15-adv-diff}, with the aim of obtaining order $h^{(k+2)}$ estimates for polynomials of degree $k$ in advection-dominated regimes. We also aim to extend these ideas to time-dependent advection diffusion equations.
\section{Acknowledgements}
The author would want to thank Prof. J\'er\^ome Droniou and Prof. Barry Koren for the discussions and advice, which helped improve the presentation of the paper.
\section{Appendix}
In this section, we present without proof two lemmas from \cite{VDM11-adv-diff,EGH10-SUSHI} which are used in the convergence proof of Theorem \ref{th:conv}.
\begin{lemma}[Discrete Sobolev inequality] \label{lem:discSob}
	Let $\mesh_h$ be an admissible discretisation of $\O$ satisfying assumptions \ref{assum. mesh_star} and \ref{assum. mesh_reg}. Let $\theta > 0$ and such that $\theta<\frac{d_{K,\sigma}}{d_{K',\sigma}} < \theta^{-1}$  for all $\sigma \in \edges_{h,int}$. Let $r =\frac{2d}{d-2}$ if $d>2$ and $r < \infty$ if $d = 2$. Then there exists a real positive constant $C$ that only depends on $\O, \theta$ and $r$ such that, for all $q_h \in X_{\disc_h}$, we have $\norm{q_h}{L^r(\O)} \leq C\norm{q_h}{1,\disc_h}$ .
\end{lemma}
\begin{lemma}[Discrete Rellisch theorem]\label{lem:discRel} Let $\Lambda$ be a diffusion tensor satisfying hypothesis \ref{assum.Diff}. Let $(\mesh_h)_{h\rightarrow0}$ be a family of admissible discretizations of $\O$ with mesh size $h$ tending to 0 and satisfying the regularity assumptions \ref{assum. mesh_star} and \ref{assum. mesh_reg}. Let $c_h \in X_{\disc_h}$ be a numerical scalar field such that $\norm{c_h}{1,\disc_h}$ remains bounded	as $h \rightarrow 0$. Then there exists a scalar field $c\in H^1(\O)$ such that, up to a subsequence as $h \rightarrow 0$, the following hold:
	\begin{enumerate}
		\item $c_h \rightarrow c$ in $L^r(\O)$ for all $r < \frac{2d}{d-2}$;
		\item $\ograd_{\disc_h} c_h \rightarrow \nabla c$ weakly in $L^2(\O)^d$.
	\end{enumerate}
\end{lemma}
	\bibliographystyle{abbrv}
	\bibliography{2d-fully_local}
\end{document}

%% file: macros.tex
\allowdisplaybreaks

\definecolor{labelkey}{rgb}{0.6,0,1}

\usepackage[normalem]{ulem}
\newcounter{corr}
\definecolor{violet}{rgb}{0.580,0.,0.827}
\newcommand{\corr}[3]{\typeout{Warning : a correction remains in page
\thepage}
				\stepcounter{corr}        
				{\color{blue}\ifmmode\text{\,\sout{\ensuremath{#1}}\,}\else\sout{#1}\fi}
        {\color{red}#2}
        {\color{violet} \fbox{\thecorr}#3}}

\newcounter{cst}
\catcode`\@=11
\def\ctel#1{C_{\refstepcounter{cst}\@bsphack
\protected@write\@auxout{}%
           {\string\newlabel{#1}{{\thecst}{\thepage}}}\thecst}}
\catcode`\@=12

\newcounter{cexp}
\catcode`\@=11
\def\terml#1{T_{\refstepcounter{cexp}\@bsphack
\protected@write\@auxout{}%
           {\string\newlabel{#1}{{\thecexp}{\thepage}}}\thecexp}}
\catcode`\@=12

\newcommand{\mathbi}[1]{{\boldsymbol #1}}

\newcommand{\eop}{{\unskip\nobreak\hfil\penalty50
           \hskip2em\hbox{}\nobreak\hfil\mbox{\rule{1ex}{1ex} \qquad}
   \parfillskip=0pt
   \finalhyphendemerits=0\par\medskip}}
\renewenvironment{proof}[1][]{\noindent {\bf Proof#1. } }{\eop}

\newtheorem{theorem}{Theorem}[section]

\newtheorem{lemma}[theorem]{Lemma} 
\newtheorem{definition}[theorem]{Definition}

\definecolor{shadecolor}{gray}{0.92}
\definecolor{TFFrameColor}{gray}{0.92}
\definecolor{TFTitleColor}{rgb}{0,0,0}

\newcommand{\ba}{\begin{array}{llll}   }
\newcommand{\bac}{\begin{array}{c}}
\newcommand{\bari}{\begin{array}{r}}
\newcommand{\ea}{\end{array}}

\newcommand{\ban}{\begin{array}{llll}}
\newcommand{\ean}{\end{array}}

\newcommand{\be}{\begin{equation}}
\newcommand{\ee}{\end{equation}}

\newcommand{\beqsys }{\beqtab \left \{ \begin{array}{l}}
\newcommand{\eeqsys }{\end{array} \right . \eeqtab }

\newcommand{\benum}{\begin{enumerate}}
\newcommand{\eenum}{\end{enumerate}}

\newcommand{\beqtab}{\begin{eqnarray}} 
\newcommand{\eeqtab}{\end{eqnarray}}

\newcommand{\combineln}[2]{{\rm(\textbf{#1#2})}}

\newcommand{\centercv}{\x_\cv}                         

\newcommand{\cv}{K}

\newcommand{\disc}{{\mathcal D}}

\renewcommand{\div}{{\mathop{\rm div}}}
\newcommand{\divg}{\div}

\newcommand{\edge}{\sigma}

\newcommand{\edges}{{\mathcal E}}              
\newcommand{\edgescv}{{{\edges}_\cv}}

\newcommand{\grad}{\nabla}

\newcommand{\mesh}{{\mathcal M}}

\newcommand{\norm}[2]{\| #1 \|_{#2}}

\renewcommand{\O}{\Omega}

\newcommand{\phy}{\varphi}

\newcommand{\R}{\mathbb R}

\newcommand{\V}{\mathbf{V}}

\newcommand{\x}{\mathbi{x}}

\renewcommand{\norm}[2]{\left\Vert#1\right\Vert_{#2}}

\def\ograd{\overline{\grad}}

\def\gradD{\nabla_\disc}

\usepackage{xparse}
\DeclareDocumentCommand{\RPiD}{ O{\disc} O{,0} }{\Pi_{#1}(X_{#1#2})}

\def\Fdof#1{{\bm{\mathcal{F}}(#1,\R)}}
\makeatletter
\def\Fdof{\@ifnextchar[{\@with}{\@without}}
\def\@with[#1]#2{{\bm{\mathcal{F}}(#2;#1)}}
\def\@without#1{{\bm{\mathcal{F}}(#1,\R)}}
\makeatother

\def\RT0{\mathbb{RT}_0}